\documentclass[12pt,a4paper,reqno]{amsart}
\usepackage[utf8]{inputenc}
\usepackage[english]{babel}
\usepackage[T1]{fontenc}
\usepackage{amsmath}
\usepackage{amsfonts}
\usepackage{amsthm}
\usepackage{amssymb}
\usepackage{graphicx}
\usepackage{lmodern}
\usepackage{subcaption}
\usepackage{caption}
\usepackage{hyperref}
\usepackage{titling}
\usepackage{authblk}
\usepackage{mathabx}
\usepackage{color}
\usepackage{calrsfs}
\usepackage{comment}
\usepackage{stmaryrd}
 \usepackage{amsaddr}
\usepackage{multirow}
\usepackage{array}
\usepackage[toc,page]{appendix}
\usepackage[left=2cm,right=2cm,top=2cm,bottom=2cm]{geometry}
\title{Fully spectral scheme for the linear  BGK equation on the whole space.}

\author{Bastien Grosse}
\address{Nantes Université, Laboratoire de mathématiques Jean Leray, 2 rue de la Houssinière, 44322 Nantes Cedex 3}

\newtheorem{theoreme}{Theorem}
\newtheorem{lemme}{Lemma}
\newtheorem{proposition}{Proposition}
\newtheorem{conjecture}{Conjecture}

\newtheorem{remark}{Remark}

\usepackage{mathtools, stmaryrd}
\usepackage{xparse} \DeclarePairedDelimiterX{\Iintv}[1]{\llbracket}{\rrbracket}{\iintvargs{#1}}
\NewDocumentCommand{\iintvargs}{>{\SplitArgument{1}{,}}m}
{\iintvargsaux#1} %
\NewDocumentCommand{\iintvargsaux}{mm} {#1\mkern1.5mu,\mkern1.5mu#2}

\newcounter{boxlblcounter}  
\newenvironment{boxlabel}
  {\begin{list}
    {\arabic{boxlblcounter}}
    {\usecounter{boxlblcounter}
     \setlength{\labelwidth}{3em}
     \setlength{\labelsep}{0em}
     \setlength{\itemsep}{2pt}
     \setlength{\leftmargin}{1cm}
     \setlength{\rightmargin}{1cm}
     \setlength{\itemindent}{0em} 
     
    }
  }
{\end{list}}

\numberwithin{equation}{section}
\numberwithin{theoreme}{section}
\numberwithin{lemme}{section}

\numberwithin{proposition}{section}
\numberwithin{algo}{section}
\numberwithin{definition}{section}
\numberwithin{remark}{section}

\begin{document}
\maketitle

\begin{abstract}
In this article, we design a fully spectral method in both space and velocity for a linear inhomogeneous kinetic equation with mass and energy conservation. We focus on  the linear BGK equation with a confining potential $\phi$, even if the method could be applied to different collision operators. It is based upon the projection on Hermite polynomials in velocity and orthonormal polynomials with respect to the weight $e^{-\phi}$ in space. The potential $\phi$ is assumed to be a polynomial. It is, to the author's knowledge, the first scheme which preserves hypocoercive behavior in addition to several invariants.  These different properties are illustrated numerically on both quadratic and double well potential. \ \\

\end{abstract}

\tableofcontents
\newpage

\section{Introduction}

\subsection{Context}

In this work we are interested in the numerical approximation and the long time asymptotic of the linear inhomogeneous BGK equation. It describes the evolution of the probability density function $f: \mathbb{R}^{+}_t\times\mathbb{R}_x\times\mathbb{R}_v\mapsto\mathbb{R}^{+}$ of a large particles system confined by a smooth stationnary potential $\phi :  \mathbb{R}_x \mapsto \mathbb{R}$ such that $e^{-\phi}\in L^1(\mathbb{R})$. The unknown $f$ verifies

\begin{equation}
\left\lbrace
\begin{array}{rll}
\partial_{t}f  +v\partial_{x} f - \partial_{x}\phi \partial_{v} f  &= &  \mathcal{L}(f)  , \\
f(0,.,.)   &= & f_{0} \in L^{2}(\mathbb{R}\times\mathbb{R},\mathcal{M}(x,v)^{-1}dxdv).
\end{array}\right.
\label{equation/edp1}
\end{equation}

Here, $\mu(v) = \frac{1}{\sqrt{2\pi}} e^{-v^{2}/2}$ is the Gaussian density and $\mathcal{M}(x,v)=\mu(v) e^{-\phi(x)}$ is the total Maxwellian. The collision operator $\mathcal{L}$ is a linearization of the BGK operator:

\begin{equation}
\mathcal{L}(f) = -\left(f -\left(\int_{\mathbb{R}} f  dv\right) \mu(v)  - \left(\int_{\mathbb{R}} f v  dv\right) v \mu(v) - \left(\int_{\mathbb{R}} f\frac{v^{2}-1}{\sqrt{2}}  dv\right) \frac{v^{2}-1}{\sqrt{2}} \mu(v) \right)
\end{equation}

 The solution $f$ is at each instant $t>0$ in $ L^{2}(\mathbb{R}\times\mathbb{R},\mathcal{M}(x,v)^{-1}dxdv)$.  We mention that $f$ being a perturbation around equilibrium, it is possibly unsigned.

Under minimal assumptions, equilibria of (\ref{equation/edp1}) have been classified completely for the first time in \cite{macroscopicmode}, as well as for a large class of linear kinetic equations with several moments conservation. These so-called \textit{special macroscopic modes} are the solutions $F$  of (\ref{equation/edp1}) which minimize entropy, \textit{i.e} which lie in the kernel of $\mathcal{L}$, thus satisfying

\begin{equation}
\left\lbrace
\begin{array}{rll}
\partial_{t}F +v\partial_{x} F - \partial_{x}\phi \partial_{v} F  &= & 0, \\
  F -\left(\int_{\mathbb{R}} F  dv\right) \mu(v)  - \left(\int_{\mathbb{R}} F v  dv\right) v \mu(v) - \left(\int_{\mathbb{R}} F\frac{v^{2}-1}{\sqrt{2}}  dv\right) \frac{v^{2}-1}{\sqrt{2}} \mu(v) &= & 0. \\
\end{array}\right.
\label{equation/sys}
\end{equation}

The space of solutions depends strongly on the symmetries of the potential. More precisely, in the general multi-dimensional setting ($x\in\mathbb{R}^{n}$), the symmetries of the potential are quantified by the number of directions in which $\phi$ is \textit{harmonic}. More precisely, define the vector space $E_\phi = Span(\{\nabla \phi(x) - x, \ | \ x\in\mathbb{R}^n  \})$ and let $d_\phi$ be its dimension. The potential can be fully non-harmonic ($d_\phi = 0$), partially harmonic ($0<d_\phi <n$) or fully harmonic ($d_\phi = n$). One can construct a basis of the special macroscopic modes by studying the potential only.

In our one-dimensional setting ($n=1$), this study is trivial: $\phi$ is ever fully harmonic or fully non-harmonic. In either cases, a basis for the steady state contains:
\begin{itemize}
\item The Maxwellian  $ \mathcal{M}(x,v)$ ; 
\item The energy mode $\mathcal{E}(x,v) = (\frac{1}{2}(v^{2}-1)+\phi(x)-<\phi>)   \mathcal{M}(x,v)$.
\end{itemize}

When $\phi$ is fully harmonic, the basis also contains oscillatory modes:

\begin{itemize}
\item $f_{+}(t,x,v)= (x\cos(t)+v\sin(t))\mathcal{M}(x,v)$
\item $f_{-}(t,x,v) = (v\cos(t)-x\sin(t))\mathcal{M}(x,v)$
\item $g_{+}(t,x,v) = (xv\cos(2t) + \frac{1}{2}(x^{2}-v^{2})\sin(2t))\mathcal{M}(x,v)$
\item $g_{-}(t,x,v) = (\frac{1}{2}(x^{2}-v^{2})\cos(2t) -xv \sin(2t))\mathcal{M}(x,v)$
\end{itemize}

Each of these modes are in duality of a conservation law: the first two modes correspond respectively to the mass and energy conservation. The other one correspond to less explicit conservations.

In \cite{macroscopicmode}, the question of the convergence toward minimizers of entropy was answered by using hypocoercivity methods. Let us first introduce the suitable minimizer $f_{\infty}$  and the perturbation $h$ around it. They are defined by 

$$
f_{\infty}= \alpha \mathcal{M} + \beta \mathcal{E} + \gamma_{+} f_{+} + \gamma_{-}f_{-} +\delta_{+} g_{+} + \delta_{-}g_{-}
$$

\noindent where the coefficients $\alpha,\beta,\gamma_{+},\gamma_{-},\delta_{+},\delta_{-}$ are computed so that the perturbation 

$$
h = \dfrac{f-f_{\infty}}{\mathcal{M}}
$$

\noindent satisfies the hypothesis of Proposition \ref{proposition/1} at $t=0$. The perturbation belongs to the weighted space $L^{2}(\mathbb{R}\times\mathbb{R},\mathcal{M}(x,v)dxdv)$ which is endowed with the natural norm

$$
\| f \|_{L^{2}(\mathcal{M})} := \left(  \int_{\mathbb{R}\times\mathbb{R}} f^{2}(x,v) \mathcal{M}(x,v)dxdv   \right)^{1/2}.
$$

\noindent The hypocoercivity result takes the form:

\begin{theoreme}[\cite{macroscopicmode}]
There exists two positive constants $C,\kappa$ such that for any solution $f\in \mathcal{C}(\mathbb{R}^{+},L^{2}(\mathbb{R}\times\mathbb{R},\mathcal{M}(x,v)dxdv))$ of (\ref{equation/edp1}) with initial condition $f_{0}\in L^{2}(\mathbb{R}\times\mathbb{R},\mathcal{M}(x,v)dxdv)$, 

$$
\forall t\geq 0, \ \ \|h(t)\|_{ L^{2}(\mathcal{M})} \leq C e^{-\kappa t} \|h(t=0)\|_{ L^{2}(\mathcal{M})}.
$$
\label{theoreme/hypofred}
\end{theoreme}

In the following, we will propose a fully spectral method to approximate the perturbation $h$ which satisfies by a direct computation from (\ref{equation/edp1}):

\begin{equation}
\left\lbrace
\begin{array}{rll}
\partial_{t}h  +v\partial_{x} h - \partial_{x}\phi \partial_{v} h  &= & Lh, \\
h(0,.,.)   & = & h_{0} \in L^{2}(\mathbb{R}\times\mathbb{R},\mathcal{M}(x,v)dxdv).
\end{array}\right.
\label{equation/boltzmannequation}
\end{equation}

The right hand side is defined by 

$$
Lf := -\left(h-\left(\int_{\mathbb{R}} h \mu dv\right)  - \left(\int_{\mathbb{R}} h v \mu dv\right) v - \left(\int_{\mathbb{R}}h \frac{v^{2}-1}{\sqrt{2}} \mu dv\right) \frac{v^{2}-1}{\sqrt{2}} \right)
$$

and is nothing more than the orthogonal projection onto $Span\{\mu,v\mu,v^2 \mu\}^\perp$.

Our main purpose is to propose discrete analogues of this hypocoercivity result together with discrete analogues of the conservation laws. We mention that our numerical method will not introduce any artificial boundary condition.  In this work, only the linearized BGK operator is considered, but more general linear collision operators may be considered as long as their representation in the Hermite basis is banded.

\subsection{Notations and assumptions}

 If $a,b\in\mathbb{R}$, the notation "$ a \lesssim b $" means that there exists a constant $k>0$ such that $a\leq kb$. We denote for convenience 

\begin{equation}
\delta_{k\geq 3} = 
\left\lbrace\begin{array}{ccc}
1 & \mbox{ if } & k\geq 3 \\
0 & \mbox{ else } &
\end{array}\right.
\end{equation}

In what follows, the potential $\phi$ is an even polynomial of degree $2m$, with $m\geq 1$. This assumption seems restrictive, but it allows anyway a variety of interesting case, such as harmonic and double well potential (see below). It is also essential for our discretization. 

Let  

$$\rho: = e^{-\phi}.$$

As $\phi$ is even, it is direct that $\rho$ is centered:

$$
\int_{\mathbb{R}}x \rho dx =0.
$$

We  make two other assumptions on $\rho $:

$$
\int_{\mathbb{R}} \rho dx =1 ; \ \ \int_{\mathbb{R}} \partial_{x}^{2}\phi \rho dx = 1.
$$

Indeed, if it is not the case, just replace $\phi$ by $\tilde{\phi}(x) = \phi(\gamma x )+\ln(c)$ where

$$
c = \sqrt{\int_{\mathbb{R}} e^{-\phi}dx \int_{\mathbb{R}} \partial^{2}_{x}\phi e^{-\phi}dx  }; \ \ \gamma = \sqrt{\dfrac{\int_{\mathbb{R}} e^{-\phi}dx}{ \int_{\mathbb{R}} \partial^{2}_{x}\phi e^{-\phi}dx}}.
$$

With our choice of normalization, the \textit{harmonic case} corresponds to $\phi(x) = \frac{1}{2}(x^{2}+\ln(2\pi))$. We will refer to the \textit{double well case} when $\phi(x) =(x-1)^{2}(x+1)^{2}$ before normalization.

Let $L^{2}(\rho)$ be the space of square integrable function with respect to the measure $\rho(x)dx$ on $\mathbb{R}$. The $L^{2}(\rho)$-norm is denoted

$$
\| f \|:= \left(\int_{\mathbb{R}} f(x)^{2}\rho(x)dx\right)^{1/2}.
$$
\noindent The  scalar product in $L^{2}(\rho)$  is defined by

$$
<f,g> := \int_\mathbb{R} f g \rho dx 
$$
\noindent The mean of $f\in L^{2}(\rho)$ under the mesure $\rho(x)dx$ is denoted without ambiguity by

$$ 
<f> := \int_{\mathbb{R}} f(x) \rho(x) dx.
$$

\noindent Let $H^{1}(\rho)$ be the Sobolev space consisting  of  the functions $f\in L^{1}_{loc}(\mathbb{R})$ such that
$$
\|f\| +  \| \partial_{x}f \| < \infty.
$$

\noindent Here, $\partial_{x}f$ is the  derivative of $f$ in the  distributional sense. The adjoint of $\partial_{x}$ in the space $L^{2}(\rho)$ is the operator $\partial_{x}^{*} = -\partial_{x}+\partial_{x}\phi$. The weighted laplacian, sometimes called the Witten laplacian, is the self-adjoint differential operator of $L^{2}(\rho)$ defined by $\partial_{x}^{*}\partial_{x}$. We then define  operator $\Omega$ by 

\begin{equation}
\Omega := \partial_{x}^{*}\partial_{x} + 1.
\label{equation/Omega}
\end{equation}

\noindent The weighted Laplacian replace in our geometry the usual Laplacian which is self-adjoint in the flat, unweighted space $L^{2}(\mathbb{R})$. Operator $\Omega$ plays a central role in the proof of hypocoercivity since among other roles, it is used to recover the missing dissipation on the first three Hermite modes in velocity.

The Gaussian density in velocity is

\begin{equation}
\mu(v) = \frac{1}{\sqrt{2\pi}} e^{-\frac{v^{2}}{2}}
\end{equation}

 and the total Maxwellian is 
 
\begin{equation}
\mathcal{M}(x,v) = \mu(v)\rho(x).
\end{equation}

The space composed of square integrable functions with respect to the measure $\mathcal{M}(x,v)dxdv$ is denoted by $L^{2}(\mathcal{M})$ and is endowed with the norm

$$
\| f\|_{L^{2}(\mathcal{M})} := \left(  \int_{\mathbb{R}\times\mathbb{R}}  f(x,v)^{2}\mathcal{M}(x,v)dxdv \right)^{1/2}.
$$

We denote by $\mathbb{P}_n := \mathbb{R}_{n}[X]$  the vector space of polynomials with real coefficients and of degree less than $n$. Polynomials will be identified with the associated polynomial functions.

The sequence of orthonormal polynomials with respect to the weight $\rho$ is denoted by $(\tilde{P}_{n})_{n\in\mathbb{N}}$. These polynomials are defined precisely in Annex \ref{section/orthonormalpolynomials}. The sequence of orthonormal polynomials wit respect to the weight $\mu$ is denoted by   $(\tilde{H}_{k})_{k\in\mathbb{N}}$. These are the Hermite orthonormal polynomials (see Annex \ref{section/orthonormalpolynomials}). Notice that when $\phi$ is fully harmonic ($\phi(x)=\frac{1}{2}(x^2+\ln(2\pi))$), the two functions $\rho$ and $\mu$ are the same. Hence, the two sequences of orthonormal polynomials introduced above are the same only when $\phi$ is fully harmonic. 

The orthogonal projection  $\Pi_{\mathbb{P}_N}f$ of a function $f\in L^{2}(\rho)$ on $\mathbb{P}_N$ is given by

\begin{equation}
\Pi_{\mathbb{P}_N}f = \sum_{k=0}^{N} \int_{\mathbb{R}} f \tilde{P}_{k} \rho dy  \tilde{P}_{k}.
\end{equation}

\subsection{Main results}

The scheme (see (\ref{equation/schema_discret_1}),(\ref{equation/Cktilde}),(\ref{equation/htilde}) below)  designed for Equation (\ref{equation/boltzmannequation})  is a fully spectral scheme in both space and velocity, and is implicit in time. The solution is projected on the Hermite polynomials in velocity, and is projected on the polynomials $(\tilde{P}_{n})_{n\in\mathbb{N}}$ in space.  We prove  two main results in the sequel:

\begin{itemize}
\item The scheme and its semi-discrete version (discretization on $x$ and $v$ only) present conservation laws which are analogous to the one of the continuous setting (Proposition \ref{proposition/2} and Proposition \ref{proposition/3}).
\item The semi-discrete scheme preserves hypocoercivity, i.e we adapt the work of \cite{macroscopicmode} to prove Theorem \ref{theoreme/main} below, which is the  analogue of Theorem \ref{theoreme/hypofred} in the discrete setting.
\end{itemize}

\begin{theoreme}
Let $N\geq deg(\phi)$. There are two positive constants $\omega_{N},\lambda_{N}$ such that for    every solution $\tilde{h}$ of the semi-discrete scheme (\ref{equation/schema_discret_1}),(\ref{equation/Cktilde}),(\ref{equation/htilde}), 

$$
\forall t\geq 0, \ \ \|\tilde{h} \|_{L^{2}(\mathcal{M})} \leq \omega_{N} e^{-\lambda_{N} t} \|\tilde{h}(t=0)\|_{L^{2}(\mathcal{M})} .
$$

The constants $\omega_{N},\lambda_{N}$ depend only on  $N$ and $\phi$.
\label{theoreme/main}
\end{theoreme}

Up to our knowledge, this is the first scheme  preserving both hypocoercivity and several conservation laws for a kinetic equation on the real line. It strongly relies on the continuous result of \cite{macroscopicmode}.

\subsection{Brief review of litterature and main features on the present work.}

Preserving structures in the discrete setting is crucial for observing features such as numerical hypocoercivity. For the Kolmogorov equation on the whole space, a finite difference was introduced in \cite{zuazua}. It preserves polynomial decay estimates of the solution. No practical implementation is proposed.
  For the spatialy inhomogeneous Fokker-Planck and linear Boltzmann equations on the torus, several methods have been developped. In \cite{dujardin2020coercivity}, authors build  a finite difference scheme only for the Fokker-Planck equation and in \cite{bessemoulin2020hypocoercivity}, an asymptotic-preserving finite volume scheme is built for both equations. In these two papers, authors have to choose a finite interval for the velocity since it is not possible to implement a scheme on an infinite number of cells. This leads them to build an appropriate discrete Maxwellian and to choose boundary conditions.  Later, in \cite{blaustein2024discrete}, this difficulty was by-passed by first projecting the density on the Hermite basis in velocity. Moreover, authors were able to deal with a non-zero potential. The space discretization was achieved by a finite volume scheme. Fluxes were chosen so that the scheme was consistant with the PDE and the mass was conserved. The same ideas are not practical for our work: indeed, we did not find a flux which ensures conservation of mass and energy (and more if the potential is harmonic) and such that the scheme is consistent with the equation. Moreover, a finite volume scheme would not be implementable on $\mathbb{R}$. 
  
Our approach is purely spectral. We project first the density on Hermite polynomials in velocity, then project the Hermite coefficients on suitable orthonormal polynomials in space. It is thus implementable. This discretization is the first to  preserve every invariants depending on the harmonicity of the potential, and features numerical hypocoercivity. Contrary to \cite{blaustein2024discrete}, hypocercivity is established on truncated expansions rather than on the whole expansion. This raises the question of dependance of the hypocoercivity constants on the truncation parameters.

\subsection{Outline of the paper}

In Section \ref{section/1}, we project Equation (\ref{equation/boltzmannequation}) on the Hermite polynomials $(\tilde{H}_{k})_{k\in\mathbb{N}}$ in velocity. We thus  get an infinite system of PDEs (\ref{equation/systeme_hyperbolique}) satisfied by the coefficients of the solution in this basis and use it to exhibit the conservation laws.

 In Section \ref{section/2}, we project System (\ref{equation/systeme_hyperbolique}) on the orthonormal polynomials $(\tilde{P}_{n})_{n\in\mathbb{N}}$ to get a new  system of ordinary differential equations (\ref{equation/Ckn}). At this step, the perturbation $h$ is expanded as
 
$$
h(t,x,v) = \sum_{k=0}^{+\infty} \sum_{n=0}^{+\infty} C_{k,n}(t) \tilde{P}_n(x) \tilde{H}_k(v)
$$ 

and the coefficients  $C_{k,n}$ satisfy the system: $\forall k,n\in\mathbb{N}$,

\begin{equation*}
\dfrac{d}{dt} C_{k,n}(t) = \sqrt{k+1} \sum_{r=0}^{\infty}<\tilde{P}_{r},\partial_{x}\tilde{P}_{n}>C_{k+1,r}(t) - \sqrt{k} \sum_{r=0}^{\infty}<\tilde{P}_{n},\partial_{x}\tilde{P}_{r}>C_{k-1,r}(t) - \delta_{k\geq 3} C_{k,n}(t).
\end{equation*}

The main idea is to  fix $K,N\in\mathbb{N}$ and approximate $h$ by 

$$
\tilde{h}(t,x,v) = \sum_{k=0}^{K} \sum_{n=0}^{N} \tilde{C}_{k,n}(t) \tilde{P}_n(x) \tilde{H}_k(v).
$$ 

\noindent where the coefficients  $\tilde{C}_{k,n}$ satisfy the system: $\forall 0\leq k\leq K,\forall 0\leq n\leq N$,

$$
\dfrac{d}{dt} \tilde{C}_{k,n}(t)  = \sqrt{k+1} \sum_{r=0}^{N}<\tilde{P}_{r},\partial_{x}\tilde{P}_{n}>\tilde{C}_{k+1,r}(t) - \sqrt{k} \sum_{r=0}^{N}<\tilde{P}_{n},\partial_{x}\tilde{P}_{r}>\tilde{C}_{k-1,r}(t) - \delta_{k\geq 3} \tilde{C}_{k,n}(t) 
$$

and $\tilde{C}_{K+1,n}=0$ for all $0\leq n \leq N$.

We then state the discrete conservation laws. The key property for this result is that the projection $\Pi_{\mathbb{P}_N}$ is self-adjoint in $L^2(\rho)$, so for every $f\in L^{2}(\rho)$ and $P\in \mathbb{P}_N$, it holds that
 
 $$ <\Pi_{\mathbb{P}_N}f,P>=<f,P>. $$

 In Section \ref{section/3}, we use an implicit Euler scheme to discretize the fully projected system (\ref{equation/schema_discret_2}) in time, and exhibit once again every discrete conservation laws. Section \ref{section/4} is devoted to the proof of Theorem \ref{theoreme/main}. The proof proceeds in two main steps. The first one is to build an entropy functional $\mathcal{H}_{1}$, which is sufficient to prove hypocoercivity in the harmonic case. The second one is to complete this first entropy to get an entropy functional $\mathcal{H}_{2}$ which is  used to prove hypocoercivity in the non-harmonic case. We illustrate numerically some properties of our scheme in Section \ref{section/5}. Conclusions are gathered in Section \ref{section/6} and Annex \ref{section/orthonormalpolynomials} contains reminders on orthonormal polynomials.

\ \\
\textbf{Acknowledgments}. The author would like to thank Mehdi Badsi and Frédéric Hérau for their numerous comments on all  aspects of this work.

\section{Continuous setting}

 \label{section/1}

\subsection{Decomposition on Hermite polynomials in velocity}

The first step in constructing our scheme is to project  Equation (\ref{equation/boltzmannequation}) onto the basis of the normalized Hermite polynomials $(\tilde{H}_{k})_{k\in\mathbb{N}}$, which form an orthonormal basis of $L^{2}(\mathbb{R},\mu dv)$ equipped with the usual scalar product (see Annex \ref{section/orthonormalpolynomials}). This is possible because $\int_{\mathbb{R}} h^{2} \mu(v)dv <\infty$ for almost all $x\in\mathbb{R}$ and almost all $t>0$. The coefficients of $h$ in this basis are the time and space functions $C_{k}(t,x)$, hence

\begin{equation}
 h(t,x,v) = \sum_{k=0}^{\infty} C_{k}(t,x) \tilde{H}_{k}(v).
\end{equation}

By projecting the equation on the Hermite polynomials $\tilde{H}_{k},k\in\mathbb{N}$, we show that the coefficients $C_{k},k\in\mathbb{N}$ verify the following system:

\begin{equation}
\ \partial_{t}C_{k} = \sqrt{k+1}\partial_{x}^{*}C_{k+1} - \sqrt{k}\partial_{x} C_{k-1} -\delta_{k\geq 3} C_{k}.  \ \ \  \mathbf{(E_{k})} 
\label{equation/systeme_hyperbolique}
\end{equation}

Parseval formula and the monotone convergence theorem show that 

$$
\forall t\geq 0, \ \ \| h(t)\|_{L^{2}(\mathcal{M})}^{2} = \sum_{k=0}^{\infty}  \| C_{k}(t)  \|^{2} 
$$

\noindent and that $C_{k}\in L^{2}(\rho)$ for all $k\in\mathbb{N}$.

\subsection{Conservation laws}

In this section, we establish conservation laws identified in \cite{macroscopicmode} by using the framework given by  System (\ref{equation/systeme_hyperbolique}). Let us define the quantities $r,m$ and $e$ (local mass, local momentum, local kinetic energy):

\begin{eqnarray*}
r(t,x) & := & \int_{\mathbb{R}} h(t,x,v) \tilde{H}_{0}(v) \mu(v)dv \ = C_{0}(t,x), \\
m(t,x) & := & \int_{\mathbb{R}} h(t,x,v)  \tilde{H}_{1}(v) \mu(v)dv \ =C_{1}(t,x), \\\
e(t,x) & := & \int_{\mathbb{R}} h(t,x,v)   \tilde{H}_{2}(v) \mu(v) dv = C_{2}(t,x).
\end{eqnarray*}

Proposition \ref{proposition/1} gives all the conservation laws for both harmonic and non-harmonic potentials.  We highlight that these could be directly deduced from the Boltzmann BGK equation (\ref{equation/boltzmannequation}), by using similar computations. We prove them below using the framework of System (\ref{equation/systeme_hyperbolique}) to show the algebraic properties at play. 

\begin{proposition}[\cite{macroscopicmode}]

Suppose that $h\in C^{0}(\mathbb{R}^{+},L^{2}(\mathcal{M}))$ is a solution of the linear Boltzmann BGK equation (\ref{equation/boltzmannequation}).

\begin{boxlabel}
\item  In the case of a general potential, and if at $t=0$, 

\begin{equation}
<r>=0
\label{equation/1}
\end{equation}

\noindent then this is true at all times $t\geq 0$. In the same way, if at $t=0$,
 
\begin{equation}
\dfrac{1}{\sqrt{2}}<e> +<\phi r>=0
\label{equation/2}
\end{equation}

\noindent then it is true at all times $t\geq 0$.

\item If $\phi$ is harmonic, and if at $t=0$,

\begin{equation}
<rx>=0 \mbox{ and } <m>=0,
\label{equation/3}
\end{equation}

\noindent then it is true at all times $t\geq  0$. In the same way, if at $t=0$,

\begin{equation}
<r>=0, \ <mx>=0 \mbox{ and } \dfrac{1}{\sqrt{2}}<e>-<\phi r> = 0
\label{equation/4}
\end{equation}

\noindent then it is true at all times $t\geq 0$. 

\end{boxlabel}

\label{proposition/1}

\end{proposition}

\begin{proof}

Recall the first three equations of  System (\ref{equation/systeme_hyperbolique}):

$$
\begin{array}{cclc}
\partial_{t}C_{0} & = & \partial_{x}^{*}C_{1} & \mathbf{(E_{0})} \\
\partial_{t}C_{1} & = & \sqrt{2} \partial_{x}^{*}C_{2} - \partial_{x} C_{0} & \mathbf{(E_{1})} \\
\partial_{t}C_{2} & = & \sqrt{3} \partial_{x}^{*}C_{3} - \sqrt{2} \partial_{x}C_{1} & \mathbf{(E_{2})}
\end{array}
$$

\begin{boxlabel}

\item Equation (\ref{equation/1}) is obtained by multiplying  Equation $\mathbf{(E_{0})}$ by $\rho$, integrating and performing an integration by parts :

\begin{eqnarray*}
\dfrac{d}{dt} < r> = <\partial_{x}^{*}C_{1}, 1> = <C_{1}, \partial_{x}1> = 0.
\end{eqnarray*}

\noindent One may apply the same process to  Equation $\mathbf{(E_{2})}$. This time,   noticing  that  $<\partial_{x}^{*}C_{3},1>=<C_{3},\partial_x 1>=0$ and that $<\partial_{x}C_{1},1> = <C_{1},\partial_x^*1> = <C_{1},\partial_x\phi>=<m\partial_x\phi>$, one get

\begin{eqnarray*}
\dfrac{d}{dt}<e> & = & \sqrt{3} <\partial_{x}^{*}C_{3},1> - \sqrt{2}<\partial_{x}C_{1},1>  = -\sqrt{2}<m\partial_x\phi>
\end{eqnarray*}

Finally, multiplying  Equation $\mathbf{(E_{0})}$ by $\phi\rho$, integrating and performing an integration by parts gives

\begin{eqnarray*}
\dfrac{d}{dt} <r\phi> =  <\partial_{x}^{*}C_{1},\phi> = & <C_{1},\partial_{x}\phi>=  <m\partial_{x}\phi>.
\end{eqnarray*}

The last two identities implie  conservation of total energy (\ref{equation/2}).

\item In the harmonic case, $\phi(x)=\frac{1}{2}(x^{2}+\ln(2\pi))$, and $\partial_{x}\phi(x) = x$. By multiplying $\mathbf{(E_{0})}$ by $x\rho$, integrating and performing an integration by parts, one find

$$
\dfrac{d}{dt}<rx> = <\partial_{x}^{*}C_{1}, x> = <C_{1}, 1> = <m>.
$$

Next, one multiply   Equation $\mathbf{(E_{1})}$ by $\rho$, integrates and notice  that  \newline $<\partial_{x}^{*}C_{1},1> = <C_{1},\partial_x 1>=0$ and  $<\partial_{x}C_{0},1>  = <C_{0},\partial_x^* 1> =<rx>$ to get 

\begin{eqnarray*}
\dfrac{d}{dt}<m>  =  \sqrt{2} <\partial_{x}^{*}C_{1},1> - <\partial_{x}C_{0},1> =- <rx>.
\end{eqnarray*}

The quantities $<rx>$ and $<m>$ satisfy a first-order linear ODE system. By the Cauchy-Lipschitz Theorem and considering the initial conditions, we obtain that \newline $<rx>=<m>=0$ at all times if it is true at $t=0$.

We now move on to the last identity (\ref{equation/4}).  It was already proven above that $\dfrac{d}{dt}<e>=-\sqrt{2}<m\partial_{x}\phi>$ and   $\dfrac{d}{dt}<r\phi>=-<m\partial_{x}\phi>$. We get directly that

$$
\dfrac{d}{dt} \left( \dfrac{1}{\sqrt{2}}<e> - <r\phi> \right) = -2<mx>.
$$

Finally, multiplying  Equation  $\mathbf{(E_{1})}$ by $x\rho$, integrating  and using the facts that $ <\partial_{x}^{*}C_{2},x> = <C_2,1>=<e>$, $\partial_x^*(x)=x^2-1$ , and $<C_{0},1>=0$ gives:

\begin{eqnarray*}
\dfrac{d}{dt}<mx> & = & \sqrt{2} <\partial_{x}^{*}C_{2},x> - <\partial_{x}C_{0},x> \\
				  & = & \sqrt{2} <e> - <C_{0},x^{2}-1> \\
				  & = & \sqrt{2} <e> - <C_{0},x^{2}+\ln(2\pi)> \\
				  & = & \sqrt{2} <e> - 2<r\phi> \\
				  & = & 2 \left(\dfrac{1}{\sqrt{2}}<e> - <r\phi>\right) .
\end{eqnarray*}

Once again, if the quantities are zero at $t=0$, then   the Cauchy-Lipschitz Theorem implies that they  are zero at all times.
  
\end{boxlabel}

\end{proof}

\section{Semi-discrete spectral scheme in space and velocity}
\label{section/2}

\subsection{Projection on orthonormal polynomials in space}

To discretize System  (\ref{equation/systeme_hyperbolique}) in the space variable, we project the coefficients $C_{k}$ on the basis $(\tilde{P}_{n})_{n\in\mathbb{N}}$ of orthonormal polynomials with respect to the weight $\rho$. For all $k,n\in\mathbb{N}$, we  denote the scalar product $<C_{k},\tilde{P}_{n}>$  by $C_{k,n}(t)$. Taking the $L^{2}(\rho)$-scalar product of  Equation $\mathbf{(E_{k})}$ of System (\ref{equation/systeme_hyperbolique}) by $\tilde{P}_{n}$, we obtain that for all $k,n\in\mathbb{N}$: 

\begin{equation}
\dfrac{d}{dt} C_{k,n}(t) = \sqrt{k+1} \sum_{r=0}^{\infty}<\tilde{P}_{r},\partial_{x}\tilde{P}_{n}>C_{k+1,r}(t) - \sqrt{k} \sum_{r=0}^{\infty}<\tilde{P}_{n},\partial_{x}\tilde{P}_{r}>C_{k-1,r}(t) - \delta_{k\geq 3} C_{k,n}(t).
\label{equation/Ckn}
\end{equation}

We use  this formulation to define a semi-discrete scheme. For fixed $K,N\in\mathbb{N}$, the semi-discrete scheme consists in solving the following linear system of ODEs:

\begin{small}
\begin{equation}
\left\lbrace
\begin{array}{lll}
\dfrac{d}{dt} \tilde{C}_{k,n}(t) & = & \sqrt{k+1} \sum_{r=0}^{N}<\tilde{P}_{r},\partial_{x}\tilde{P}_{n}>\tilde{C}_{k+1,r}(t) - \sqrt{k} \sum_{r=0}^{N}<\tilde{P}_{n},\partial_{x}\tilde{P}_{r}>\tilde{C}_{k-1,r}(t) - \delta_{k\geq 3} \tilde{C}_{k,n}(t)  \\
 \tilde{C}_{k,n}(0)  & = & \int_{\mathbb{R}\times\mathbb{R}} h(0,x,v) \tilde{P}_{n}(x)\tilde{H}_{n}(v) \mathcal{M}(x,v)dxdv \ \  \ \forall k\in\Iintv{0,K},n\in\Iintv{0,N}  \\
 \tilde{C}_{-1,n}  &=&  \tilde{C}_{K+1,n}  = 0 \ \ \ \forall n\in\Iintv{0,N}. \label{equation/schema_discret_1}
\end{array}\right. 
\end{equation}
\end{small}

We define an approximation of $C_{k}$  by the formula below:

\begin{equation}
\tilde{C}_{k}(t,x) := \sum_{n=0}^{N} \tilde{C}_{k,n}(t) \tilde{P}_{n}(x).
\label{equation/Cktilde}
\end{equation}

The approximation of $h$ is then

\begin{equation}
\tilde{h}(t,x,v) := \sum_{k=0}^{K} \tilde{C}_{k}(t,x) \tilde{H}_{k}(v).
\label{equation/htilde}
\end{equation}

We can give an equivalent formulation of System (\ref{equation/schema_discret_1}). Simply multiply the equation on $\tilde{C}_{k,n}$ by $\tilde{P}_{n}$, then sum for $0\leq n \leq N$. We then obtain the equation:

\begin{equation}
\left\lbrace
\begin{array}{lll}
 \forall k \leq K, \  \partial_{t} \tilde{C}_{k} = \sqrt{k+1} \Pi_{\mathbb{P}_N} \partial_{x}^{*}\tilde{C}_{k+1} - \sqrt{k} \partial_{x} \tilde{C}_{k-1} - \delta_{k\geq 3} \tilde{C}_{k}  \ \ \ \mathbf{(\tilde{E}_{k})} \\
 \forall k \leq K, \   \tilde{C}_{k}(0,x) = \Pi_{\mathbb{P}_N}\int_{\mathbb{R}} h(0,x,v)\mu(v)dv.
  \end{array}\right.
\label{equation/schema_discret_2}
\end{equation}

While  formulation (\ref{equation/schema_discret_1}) of the scheme is used for practical implementation, the concise formulation (\ref{equation/schema_discret_2}) is more convenient in view of the theoretical analysis.

\begin{remark}
The orthogonal projection $\Pi_{\mathbb{P}_N}$ appears in the formulation (\ref{equation/schema_discret_1}) only before the operator $\partial_x^*$. Indeed, $\mathbb{P}_N$ is stable by $\partial_x$, but due to the multiplication by $\partial_x\phi$, $\mathbb{P}_N$ is not stable by $\partial_x^*$.
\end{remark}

\subsection{Discrete conservation laws}

We will now prove the conservation laws analogous to the conservation laws of the continuous model. We will have to make an assumption on the parameter $N$, 

$$
\textbf{(H)} \ \  N\geq \ deg(\phi)  
\label{assumption/H} 
$$

 Assumption \textbf{(H)} will be  essential for both the semi-discrete and the fully discrete scheme. It will also  play a key role for the hypocoercivity estimates later. Thus, from now on, we will always suppose \textbf{(H)} although we will not write it explicitely. Notice that it implies that $N\geq 2$ since $\phi$ is at least of degree $2$.
 Proposition \ref{proposition/2} identifies invariant quantities for System (\ref{equation/schema_discret_2}). Its proof is nearly identical to the proof of Proposition (\ref{proposition/1}), so we don't detail it. The major difference is the presence of the orthogonal projection $\Pi_{\mathbb{P}_N}$ in System (\ref{equation/schema_discret_2}). Then, the following algebraic properties must be used :

\begin{itemize}
\item   $<\Pi_{\mathbb{P}_N}\partial_x^* \tilde{C}_k, 1>= <\tilde{C}_k,\partial_x 1>=0$ for $k=0,1,2$ since $1\in \mathbb{P}_N$ ;
\item  $<\Pi_{\mathbb{P}_N}\partial_x^* \tilde{C}_1,\phi> =<\partial_x^* \tilde{C}_1,\phi> $ since $\phi\in \mathbb{P}_N$  (assumption (\textbf{H}));
\item $<\Pi_{\mathbb{P}_N}\partial_x^* \tilde{C}_k,x> = <\partial_x^* \tilde{C}_k,x>$ for $k=1,2$ since $x\in \mathbb{P}_N$ (assumption (\textbf{H})).
\end{itemize}

\begin{proposition}
Let $(\tilde{C}_{k,n})_{0\leq k\leq K,0\leq n \leq N}$ be the solution of  System (\ref{equation/schema_discret_2}). Let us define the following quantities:

$$
\tilde{r} := \tilde{C}_{0} ; \ \tilde{m} : = \tilde{C}_{1}; \ \tilde{e} := \tilde{C}_{2} 
$$

\begin{boxlabel}
\item  In the case of a general potential, and if at $t=0$, 

\begin{equation}
<\tilde{r}>=0
\label{equation/5}
\end{equation}

\noindent then this is true at all times $t\geq 0$. In the same way, if at $t=0$,
 
\begin{equation}
\dfrac{1}{\sqrt{2}}<\tilde{e}> +<\phi \tilde{r}>=0
\label{equation/6}
\end{equation}

\noindent then it is true at all times $t\geq0$.

\item If $\phi$ is harmonic, and at $t=0$,

\begin{equation}
<\tilde{r}x>=0 \mbox{ and } <\tilde{m}>=0,
\label{equation/7}
\end{equation}

\noindent then it is true at all times $t\geq  0$. In the same way, if at $t=0$,

\begin{equation}
<\tilde{r}>=0, \ <\tilde{m}x>=0 \mbox{ and } \dfrac{1}{\sqrt{2}}<\tilde{e}>-<\phi \tilde{r}> = 0
\label{equation/8}
\end{equation}

\noindent then it is true at all times $t\geq 0$. 

\end{boxlabel}
\label{proposition/2}
\end{proposition}

\section{Totally discrete scheme}
\label{section/3}

\subsection{Time discretization}

The variables $x$ and $v$ have already been discretized by projections on orthonormal polynomials. It remains to discretize the $t$ variable. To do this, we choose to discretize the linear ODE system (\ref{equation/schema_discret_1}) with an implicit Euler scheme since it is unconditionally stable. Let $\Delta t>0$ be the time step. We define the instant $t_{i} := i\Delta t$ for any integer $i$.

\ \\
Let $\alpha_{r,n} := \int_{\mathbb{R}} \partial_{x}P_{r}(x) P_{n}(x) \rho(x) dx, \ \forall r,n\in\Iintv{0,N}$. We construct approximations $\tilde{C}^{i}_{k,n}$ of approximations of $\tilde{C}_{k,n}(t^{i})$ for all $ k\in \Iintv{0,K},  n\in \Iintv{0,N}, i\in\mathbb{N}$ by solving the following system of linear equations inherited from (\ref{equation/schema_discret_1}):

\begin{small}
\begin{equation}
\left\lbrace
\begin{array}{ccl}
\tilde{C}_{k,n}^{0} & = &  \tilde{C}_{k,n}(0)    \\
\tilde{C}_{k,n}^{i+1} & = & \tilde{C}_{k,n}^{i} +  \Delta t ( \sqrt{k+1} \sum_{r=0}^{N} \alpha_{r,n} \tilde{C}_{k+1,n}^{i+1} - \sqrt{k} \sum_{r=0}^{N} \alpha_{n,r} \tilde{C}_{k-1,r}^{i+1} - \delta_{k\geq 3} \tilde{C}_{k,n}^{i+1} ) \ \mathbf{(\tilde{E}_{k,n}^{i})}
\end{array}
\right.
\label{equation/schema_totally_discrete}
\end{equation}
\end{small}

Remember that we set $\tilde{C}_{-1,n}=\tilde{C}_{K+1,n}=0$ for all $n\in\Iintv{0,N}$.
The approximation of $C_{k}$ at time $t_{i}$ is defined by 

$$
\tilde{C}^{i}_{k}(x) := \sum_{n=0}^{N} \tilde{C}^{i}_{k,n} \tilde{P}_{n}(x).
$$

The approximation of $h$  at time $t_{i}$ is then

$$
\tilde{h}^{i}(x,v) := \sum_{k=0}^{K} \tilde{C}^{i}_{k}(x) \tilde{H}_{k}(v).
$$

We can give an equivalent formulation for System (\ref{equation/schema_totally_discrete}). Simply multiply $\mathbf{(\tilde{E}_{k,n}^{i})}$ by $\tilde{P}_{n}$, then sum for $n\in\Iintv{0,N}$. We then obtain the equation:

\begin{equation}
  \dfrac{\tilde{C}_{k}^{i+1} - \tilde{C}_{k}^{i}}{\Delta t} = \sqrt{k+1} \Pi_{\mathbb{P}_N} \partial_{x}^{*}\tilde{C}_{k+1}^{i+1} - \sqrt{k} \partial_{x} \tilde{C}_{k-1}^{i+1} - \delta_{k\geq 3} \tilde{C}_{k}^{i+1} \ \mathbf{(\tilde{E}_{k}^{i})}
\label{equation/schema_totalement_discret_2}
\end{equation}

\subsection{Discrete conservation laws}

Let us now prove the conservation laws analogous to the conservation laws of the continuous model for this fully discrete scheme.

\begin{proposition}
Let $(\tilde{C}^{i}_{k,n})_{0\leq k\leq K,0\leq n \leq N, i\in \mathbb{N}}$ be the solution of System (\ref{equation/schema_totally_discrete}). We define the following quantities:

$$
\tilde{r}^{i} := \tilde{C}^{i}_{0} ; \ \tilde{m}^{i} : = \tilde{C}_{1}^{i}; \  \tilde{e}^{i} := \tilde{C}_{2}^{i} 
$$

\begin{boxlabel}
\item  In the case of a general potential, and if at $t_{0}=0$, 

\begin{equation}
<\tilde{r}^{0}>=0
\label{equation/9}
\end{equation}

\noindent then this is true at all times $t_{i}, \ i\in\mathbb{N}$. In the same way, if at $t_{0}=0$,
 
\begin{equation}
\dfrac{1}{\sqrt{2}}<\tilde{e}^{0}> +<\phi \tilde{r}^{0}>=0
\label{equation/10}
\end{equation}

\noindent then it is true at all times $t_{i}, \ i\in\mathbb{N}$.

\item If $\phi$ is harmonic, and at $t_{0}=0$,

\begin{equation}
<\tilde{r}^{0}x>=0 \mbox{ and } <\tilde{m}^{0}>=0
\label{equation/11}
\end{equation}

\noindent then it is true at all times $t_{i}, \ i\in\mathbb{N}$. In the same way, if at $t_{0}=0$,

\begin{equation}
<\tilde{r}^{0}>=0, \ <\tilde{m}^{0}x>=0 \mbox{ and } \dfrac{1}{\sqrt{2}}<\tilde{e}^{0}>-<\phi \tilde{r}^{0}> = 0
\label{equation/12}
\end{equation}

\noindent then it is true at all times $t_{i}, \ i\in\mathbb{N}$. 

\end{boxlabel}

\label{proposition/3}
\end{proposition}

\begin{proof}

The first three equations of  System (\ref{equation/schema_totalement_discret_2}) are:

$$
\begin{array}{cclc}
  \dfrac{\tilde{C}_{0}^{i+1} - \tilde{C}_{0}^{i}}{\Delta t} &  = &  \Pi_{\mathbb{P}_N} \partial_{x}^{*}\tilde{C}_{1}^{i+1}   &  \mathbf{(\tilde{E}_{0}^{i})} \\
  \dfrac{\tilde{C}_{1}^{i+1} - \tilde{C}_{1}^{i}}{\Delta t} & = & \sqrt{2} \Pi_{\mathbb{P}_N} \partial_{x}^{*}\tilde{C}_{2}^{i+1} -  \partial_{x} \tilde{C}_{0}^{i+1} & \mathbf{(\tilde{E}_{1}^{i})} \\
  \dfrac{\tilde{C}_{2}^{i+1} - \tilde{C}_{2}^{i}}{\Delta t} & = & \sqrt{3} \Pi_{\mathbb{P}_N} \partial_{x}^{*}\tilde{C}_{3}^{i+1} - \sqrt{2} \partial_{x} \tilde{C}_{1}^{i+1}  & \mathbf{(\tilde{E}_{2}^{i})}
\end{array}
$$

The proof is  analogous to the proof of Proposition (\ref{proposition/3}) and is left to the reader.

\end{proof}

\subsection{Exponential decay for the fully discrete scheme}

Let $M\in M_{(N+1)(K+1)}(\mathbb{R})$ be the matrix of the ODE system (\ref{equation/schema_discret_1}). According to Proposition \ref{proposition/2}, the space $\mathbb{R}^{(N+1)(K+1)}$ can be decomposed as a direct sum $\mathcal{D}^\perp\oplus\mathcal{D}$, where $\mathcal{D}$ corresponds to vector such that every invariants identified in Proposition \ref{proposition/2} is zero. It is stable by $M$, so we can define the matrix $A$ as the restriction of $M$ to $\mathcal{D}$.  After a suitable change of basis, and under the condition that every invariant is zero at $t=0$, System (\ref{equation/schema_discret_1}) reduce to a smaller ODE system

$$
\frac{d}{dt}u(t) = Au(t)
$$

Theorem \ref{theoreme/main}  shows that this system is exponentially stable, and by Lyapunov Stability Theorem implies that there exists a symetric, positive-definite matrix $P$ such that $A$ is  $\alpha$-coercive ($\alpha>0$) in the  norm $\|.\|_{P}$ associated to the scalar product $(x,y)_P := x^T Py$. The implicit Euler scheme for this problem reads

$$
\forall i\in\mathbb{N}, \ u_{i+1} + \Delta t Au_{i+1} = u_{i} .
$$

\noindent We therefore  deduce 

$$
 \|u_{i+1}\|_{P}^{2} + \Delta t (Au_{i+1},u_{i+1})_{P} = (u_{i},u_{i+1})_{P}.
$$

\noindent Using the coercivity and the Young inequality, we find that
$$
 \|u_{i+1}\|_{P}^{2} + \alpha\Delta t \|u_{i+1}\|_{P}^{2}  \leq \frac{1}{2}( \|u_{i}\|_{P}^{2}+ \|u_{i+1}\|_{P}^{2}).
$$

\noindent Hence, 
$$
 \|u_{i+1}\|_{P}^{2} \leq \frac{1}{1+2\alpha\Delta t} \|u_{i}\|_{P}^{2}
$$

\noindent and we deduce the decay estimates

$$
\forall i\in\mathbb{N}, \  \|u_{i}\|_{P}^{2} \leq \frac{1}{(1+2\alpha\Delta t)^{i}} \|u_{0}\|_{P}^{2}.
$$

\noindent If now $\Delta t = \frac{T}{I}$ for a fixed $T>0$ and $I\in\mathbb{N}^*$, then 

$$
\forall I\in\mathbb{N}, \  \|u_{I}\|_{P}^{2} \leq \frac{1}{(1+\frac{2\alpha T}{I})^{I}} \|u_{0}\|_{P}^{2}.
$$

Here, $u_{I}$ is the discrete solution at time $T$, and the bound converges as $I\to \infty$ toward $e^{-2\alpha T}\|u_{0}\|_{P}^{2}$. Note that $\alpha$ may depends on the parameter $N$ as the hypocoercivity constants in Theorem \ref{theoreme/main} do.

\section{Hypocoercivity of the Semi-Discrete Scheme}
\label{section/4}
Throughout this section, $C_{N}$ denotes a positive constant depending on $N$.

We  study the decay of the $L^2(\mathcal{M})$-norm $\|\tilde{h}(t)\|_{L^2(\mathcal{M})}$ in time. In this perspective, one may compute its time derivative. The following result is an easy consequence  of the Parseval formula and of System (\ref{equation/schema_discret_2}).

\begin{lemme}
The solution $\tilde{h}$ of the semi discrete scheme (\ref{equation/schema_discret_2}) satisfies the following identity:

$$
\forall K\in\mathbb{N}, \ \forall N\geq deg(\phi), \ \ \ \dfrac{1}{2} \dfrac{d}{dt} \| \tilde{h}(t) \|_{L^{2}(\mathcal{M})}^{2} = - \sum_{k=3}^{K} \| \tilde{C}_{k}(t) \|^{2} 
$$

\end{lemme}

The dissipations of the modes $\tilde{C}_{0},\tilde{C}_{1},\tilde{C}_{2}$ are missing in the previous lemma. Thus, it is impossible to use Gronwall's lemma directly to exhibit  the convergence to 0 with an exponential rate. This is due to the lack of coercivity of the operator $-v\partial_x + \partial_x\phi \partial_v + L$.   To recover them, we will use $L^{2}$-hypocoercivity techniques. More precisely, we build a suitable entropy functional equivalent to the $L^2(\mathcal{M})$ norm and for which we can prove exponential decay to 0. We will closely follow the strategy  proposed in \cite{macroscopicmode}. In this perspective, we will intensively make use the weighted Laplacian $\Omega$ (defined in (\ref{equation/Omega})), which satisfies the following functional inequalities.

\begin{proposition}[\cite{macroscopicmode}, \cite{poincare}]

The weighted Laplacian $\Omega$ satisfies the following inequalities.

\begin{itemize}
\item The zeroth-order strong Poincaré inequality:
\begin{equation}
\| \Omega^{-1}\partial_{x}^{2} \varphi \| \lesssim \|\varphi \|
\label{equation/poincare1}
\end{equation}
\item The Poincaré-Lions inequality:
\begin{equation}
\| \varphi - \langle \varphi \rangle \| \lesssim \| \Omega^{-1/2}\partial_{x}\varphi \| \lesssim \| \varphi - \langle \varphi \rangle \|
\label{equation/poincare2}
\end{equation}
\item The -1 order Poincaré-Lions inequality and its variant:
\begin{equation}
\| \Omega^{-1/2} (\varphi - \langle \varphi \rangle) \| \lesssim \| \Omega^{-1}\partial_{x}\varphi \| \lesssim \| \Omega^{-1/2}(\varphi - \langle \varphi \rangle) \|
\label{equation/poincare3}
\end{equation}
\begin{equation}
\| \varphi - \langle \varphi \rangle \| \lesssim \| \partial_{x}\Omega^{-1/2}\varphi \| + \| \Omega^{-1/2}\partial_{x}\varphi \| \lesssim \| \varphi - \langle \varphi \rangle \|
\label{equation/poincare4}
\end{equation}
\end{itemize}
\end{proposition}

In our discrete framework, an additional difficulty arises from the introduction of the orthogonal projection $\Pi_{\mathbb{P}_N}$. Some operators involved in the proof therefore have norms that possibly depend on $N$. We  define a constant $K_{N}$ depending only on $N$ and the potential such that:

\begin{equation}
\sup_{f \in \mathbb{P}_N, f \neq 0} \dfrac{\| \Omega^{-1/2} \Pi_{\mathbb{P}_N} \partial_{x}^{*}f \|}{\|f\|} \leq K_{N}
\label{equation/H0}
\end{equation}

\begin{equation}
\sup_{f \in \mathbb{P}_N, f \neq 0} \dfrac{\| \Omega^{-1}\partial_{x} \Pi_{\mathbb{P}_N} \partial_{x}^{*}f \|}{\|f\|} \leq K_{N},
\label{equation/H1}
\end{equation}

\begin{equation}
\sup_{f \in \mathbb{P}_N, f \neq 0} \dfrac{\| \Omega^{-1} \Pi_{\mathbb{P}_N} \partial_{x}^{*} \Pi_{\mathbb{P}_N} \partial_{x}^{*}f \|}{\|f\|} \leq K_{N},
\label{equation/H2}
\end{equation}

\begin{equation}
\sup_{f \in \mathbb{P}_N, f \neq 0} \dfrac{\| \Omega^{-1} \Pi_{\mathbb{P}_N} \partial_{x}^{*}\partial_{x}f \|}{\|f\|} \leq K_{N}.
\label{equation/H3}
\end{equation}

Note that $K_{N}$ is well-defined as these linear operators are defined on $\mathbb{P}_N$, which is finite-dimensional.
\begin{remark}
This implies that the rate of exponential decay obtained may depend on $N$ and  possibly converges to 0. If so, the scheme is inaccurate and the approximation deteriorate as $N$ increases. We don't know if $K_{N}$ is bounded in general, but we know it is bounded when $\phi$ is of degree $2$ or $4$, as shown in the simulations in Section \ref{section/5}
\end{remark}
We  also use the notations from \cite{macroscopicmode} by defining:

\begin{equation}
 r_{s} = \tilde{C}_{0} - \langle \partial_{x}\tilde{C}_{0} \rangle x - \frac{1}{2} \langle \partial_{x}^{2}\tilde{C}_{0} \rangle \xi_{2}
 \label{equation/13}
\end{equation}

\begin{equation}
 m_{s} = \tilde{C}_{1} - \langle \partial_{x}\tilde{C}_{1} \rangle x - \langle \tilde{C}_{1} \rangle 
\label{equation/14}
\end{equation}

\begin{equation}
e_{s} = \tilde{C}_{2} - \langle \tilde{C}_{2} \rangle
\label{equation/15}
\end{equation}

\begin{equation}
w = \tilde{C}_{0} - \sqrt{2} \langle \tilde{C}_{2} \rangle \phi
\label{equation/16}
\end{equation}

\begin{equation}
w_{s} = r_{s} - \sqrt{2} \langle \tilde{C}_{2} \rangle \phi_{s}
\label{equation/17}
\end{equation}

\noindent where we introduced the functions $\xi_{\phi} = \phi - \langle \phi \rangle$, $\xi_{2} = x^{2} - \langle x^{2} \rangle$, and $\phi_{s} = \xi_{\phi} - \frac{1}{2} \langle \partial_{x}^{2}\phi \rangle \xi_{2}$.

\noindent Remark that 

\begin{equation}
w_{s} = w - <\partial_{x}w>x - \dfrac{1}{2} <\partial_{x}^{2}w>\xi_{2} + \sqrt{2}<\tilde{C}_{2}><\phi>.
\label{equation/18}
\end{equation}

We  also define $\tilde{h}^{\perp}$ by 

\begin{equation}
\tilde{h}^{\perp}(t,x,v) = \sum_{k=3}^{K} \tilde{C}_{k}(t,x) \tilde{H}_{k}(v).
\end{equation}

First, we explicit evolution equations on $e_{s}, m_{s}$ and $w_{s}$ that will be useful later.

\begin{lemme} 
The functions $e_{s}, m_{s}$ and $w_{s}$ verify the following equations:

\begin{equation}
\partial_{t}e_{s} = \sqrt{3}\Pi_{\mathbb{P}_N} \partial_{x}^{*}\tilde{C}_{3} - \sqrt{2} \partial_{x}m_{s},
\label{equation/a}
\end{equation}

\begin{equation}
\partial_{t}m_{s} = \sqrt{2} \Pi_{\mathbb{P}_N}\partial_{x}^{*}\tilde{C}_{2} -\partial_{x}\tilde{C_{0}} - \sqrt{2} <\partial_{x}\Pi_{\mathbb{P}_N}\partial_{x}^{*}\tilde{C}_{2}>x + <\partial_{x}^{2}\tilde{C}_{0}>x + <\partial_{x}\tilde{C}_{0}>,
\label{equation/b}
\end{equation}

\begin{equation}
\partial_{t}m_{s}  = -\partial_{x}w_{s} + \sqrt{2}(\Pi_{\mathbb{P}_N}\partial_{x}^{*}e_{s}-  <e_{s}\partial_{x}^{2}\phi> x),
\label{equation/c}
\end{equation}

\begin{equation}
\partial_{t}w_{s} =  \Pi_{\mathbb{P}_N}\partial_{x}^{*}\tilde{C}_{1} -< \partial_{x}^{2}\phi, \tilde{C}_{1}>x - \dfrac{1}{2}<\partial_{x}\Pi_{\mathbb{P}_N}\partial_{x}^{*2}1,\tilde{C}_{1}> \xi_{2} + 2<\partial_{x}\phi,\tilde{C}_{1}> \phi_{s},
\label{equation/d}
\end{equation}

\begin{eqnarray}
 \partial_{t}^{2}w_{s} & = & \Pi_{\mathbb{P}_N}\partial_{x}^{*}(\sqrt{2}\Pi_{\mathbb{P}_N}\partial_{x}^{*}\tilde{C}_{2} -\partial_{x}\tilde{C}_{0}) -(\sqrt{2}<\partial_{x}^{3}\phi,\tilde{C}_{2}>  - <\partial_{x}^{*}\partial_{x}^{2}\phi,\tilde{C}_{0}> )x  \label{equation/e}\\
&  - &  \dfrac{1}{2}  (\sqrt{2}<\partial_{x}^{2}\Pi_{\mathbb{P}_N}\partial_{x}^{*2}1,\tilde{C}_{2}>-<\partial_{x}^{*}\partial_{x}\Pi_{\mathbb{P}_N}\partial_{x}^{*2}1,\tilde{C}_{0}>) \xi_{2} \nonumber\\
& + & 2(\sqrt{2}<\partial_{x}^{2}\phi,\tilde{C}_{2}> -<\partial_{x}^{*}\partial_{x}\phi, \tilde{C}_{0}>) \phi_{s}.\nonumber
\end{eqnarray}

\end{lemme}

\begin{proof}
\ \\

\paragraph*{Proof of (\ref{equation/a}):} First, we differentiate in time the expression of $e_{s}$, and use the scheme (\ref{equation/schema_discret_2}) :

\begin{equation*}
\partial_{t}e_{s}   =  \sqrt{3}\Pi_{\mathbb{P}_N}\partial_{x}^{*}\tilde{C}_{3} - \sqrt{2}\partial_{x}\tilde{C}_{1} - <\sqrt{3}\Pi_{\mathbb{P}_N}\partial_{x}^{*}\tilde{C}_{3} - \sqrt{2}\partial_{x}\tilde{C}_{1},1>
\end{equation*}

Then, one have $<\Pi_{\mathbb{P}_N}\partial_x^*\tilde{C}_3,1> = <\tilde{C}_3, \partial_x 1>=0$, since $1\in \mathbb{P}_N$. Thus, terms involving $\tilde{C}_1$ can be gathered to recognize the space derivative $\partial_x m_s$ :

\begin{eqnarray*} 
\partial_t e_s & = &  \sqrt{3}\Pi_{\mathbb{P}_N}\partial_{x}^{*}\tilde{C}_{3} - \sqrt{2}(\partial_{x}\tilde{C}_{1} -<\partial_{x}\tilde{C}_{1}>) \\
    & = & \sqrt{3}\Pi_{\mathbb{P}_N}\partial_{x}^{*}\tilde{C}_{3} - \sqrt{2} \partial_{x}m_{s}.
\end{eqnarray*}

\paragraph*{Proof of (\ref{equation/b}):}  We derive in time the expression of $m_{s}$, use the scheme (\ref{equation/schema_discret_2}) and perform an integration by part. Again, we use that $<\Pi_{\mathbb{P}_N}\partial_x^*\tilde{C}_2,1>=0$.

\begin{eqnarray*}
\partial_{t}m_{s}  & = & \sqrt{2}\Pi_{\mathbb{P}_N}\partial_{x}^{*}\tilde{C}_{2} - \partial_{x}\tilde{C}_{0} - <\sqrt{2}\partial_{x}\Pi_{\mathbb{P}_N}\partial_{x}^{*}\tilde{C}_{2} - \partial_{x}^{2}\tilde{C}_{0}>x -    <\sqrt{2}\Pi_{\mathbb{P}_N}\partial_{x}^{*}\tilde{C}_{2} - \partial_{x}\tilde{C}_{0}> \\
 & = & \sqrt{2}\Pi_{\mathbb{P}_N}\partial_{x}^{*}\tilde{C}_{2} - \partial_{x}\tilde{C}_{0} - \sqrt{2}<\partial_{x}\Pi_{\mathbb{P}_N}\partial_{x}^{*}\tilde{C}_{2}>x +< \partial_{x}^{2}\tilde{C}_{0}>x + <\partial_{x}\tilde{C}_{0}>. \\
\end{eqnarray*}

\paragraph*{Proof of (\ref{equation/c}): }  First, we use the scheme (\ref{equation/schema_discret_2}) to show that 

\begin{eqnarray*}
\partial_{t}\tilde{C}_{1}& =&  \sqrt{2}\Pi_{\mathbb{P}_N}\partial_{x}^{*}\tilde{C}_{2} - \partial_{x}\tilde{C}_{0} \\
& = &  \sqrt{2}\Pi_{\mathbb{P}_N}\partial_{x}^{*}e_{s} - \partial_{x}\tilde{C}_{0} + \sqrt{2} <\tilde{C}_{2}>\Pi_{\mathbb{P}_N}(\partial_{x}\phi).
\end{eqnarray*}

Since $\partial_{x}\phi\in \mathbb{P}_N$,

$$
\partial_{t}\tilde{C}_{1} = \sqrt{2}\Pi_{\mathbb{P}_N}\partial_{x}^{*}e_{s} - \partial_{x}\tilde{C}_{0} + \sqrt{2}\partial_{x}\phi <\tilde{C}_{2}> = \sqrt{2}\Pi_{\mathbb{P}_N}\partial_{x}^{*}e_{s} - \partial_{x}w.
$$

By integration with respect to $\rho dx$, we find that

$$
\dfrac{d}{dt}<\tilde{C}_{1}> = -<\partial_{x}w>.
$$

If we apply first $\partial_{x}$ and we do the same computation, we find that:

$$
\dfrac{d}{dt}<\partial_{x}\tilde{C}_{1}> = \sqrt{2} <e_{s}\partial_{x}^{2}\phi> - <\partial_{x}^{2}w>,
$$

\noindent where used above that $\partial_{x}\phi\in \mathbb{P}_N$.

These expressions allow us to compute  $\partial_{t}m_{s}$, since
 $\partial_{x}w_{s}=\partial_{x}w-<\partial_{x}^{2}w>x - <\partial_{x}w>$ from (\ref{equation/18}).

\begin{eqnarray*}
\partial_{t}m_{s} & = & \sqrt{2}\Pi_{\mathbb{P}_N}\partial_{x}^{*}e_{s} - \partial_{x}w  - \sqrt{2} <e_{s}\partial_{x}^{2}\phi> x+ <\partial_{x}^{2}w>x + <\partial_{x}w> \\
 & = & -\partial_{x}w_{s} + \sqrt{2}(\Pi_{\mathbb{P}_N}\partial_{x}^{*}e_{s}-  <e_{s}\partial_{x}^{2}\phi> x).
\end{eqnarray*}

\paragraph*{Proof of (\ref{equation/d}):} We derive in time the expression (\ref{equation/17}) of $w_{s}$ and use the scheme (\ref{equation/schema_discret_2}). Once again, $<\Pi_{\mathbb{P}_N}\partial_x^*\tilde{C}_3,1>=0$, so we have that

\begin{equation*}
\partial_{t}w_{s}   =  \Pi_{\mathbb{P}_N}\partial_{x}^{*}\tilde{C}_{1} -< \partial_{x}\Pi_{\mathbb{P}_N}\partial_{x}^{*}\tilde{C}_{1},1>x - \frac{1}{2}<\partial_{x}^{2}\Pi_{\mathbb{P}_N}\partial_{x}^{*}\tilde{C}_{1},1>\xi_{2} + 2 <\partial_x \tilde{C}_1,1>\phi_s 
\end{equation*}

We now use that $( \partial_{x}\Pi_{\mathbb{P}_N}\partial_{x}^{*})^* =  \partial_{x}\Pi_{\mathbb{P}_N}\partial_{x}^*$ and that $( \partial_{x}^2\Pi_{\mathbb{P}_N}\partial_{x}^{*})^* =  \partial_{x}\Pi_{\mathbb{P}_N}\partial_{x}^{*2}$

\begin{equation*}
 \partial_t w_s =  \Pi_{\mathbb{P}_N}\partial_{x}^{*}\tilde{C}_{1} -<\partial_{x}\Pi_{\mathbb{P}_N}\partial_{x}^{*}1 ,\tilde{C}_{1}> x- \dfrac{1}{2}<\partial_{x}\Pi_{\mathbb{P}_N}\partial_{x}^{*2}1,\tilde{C}_{1}> \xi_{2} + 2<\partial_{x}^{*}1,\tilde{C}_{1}> \phi_{s} \\
\end{equation*}

Finally, notice that $\partial_{x}\Pi_{\mathbb{P}_N}\partial_{x}^{*}1 = \partial_x^2\phi$, since owing to assumption \textbf{(H)}, $\partial_x\phi \in \mathbb{P}_N$. We thus conclude that

\begin{equation*}
 \partial_t w_s =  \Pi_{\mathbb{P}_N}\partial_{x}^{*}\tilde{C}_{1} -<\partial_{x}^2\phi ,\tilde{C}_{1}> x- \dfrac{1}{2}<\partial_{x}\Pi_{\mathbb{P}_N}\partial_{x}^{*2}1,\tilde{C}_{1}> \xi_{2} + 2<\partial_{x}\phi,\tilde{C}_{1}> \phi_{s} \\
\end{equation*}

\paragraph*{Proof of (\ref{equation/e}):}

  We derive in time the expression (\ref{equation/d}) of $\partial_{t}w_{s}$, use the scheme (\ref{equation/schema_discret_2}) and perform  integrations by part.

\begin{eqnarray*}
\partial_{t}^{2}w_{s}  & = & \Pi_{\mathbb{P}_N}\partial_{x}^{*}(\sqrt{2}\Pi_{\mathbb{P}_N}\partial_{x}^{*}\tilde{C}_{2} -\partial_{x}\tilde{C}_{0}) -< \partial_{x}^{2}\phi,\sqrt{2}\Pi_{\mathbb{P}_N}\partial_{x}^{*}\tilde{C}_{2} -\partial_{x}\tilde{C}_{0}>x\\
&  - &  \dfrac{1}{2}<\partial_{x}\Pi_{\mathbb{P}_N}\partial_{x}^{*2}1,\sqrt{2}\Pi_{\mathbb{P}_N}\partial_{x}^{*}\tilde{C}_{2} -\partial_{x}\tilde{C}_{0}> \xi_{2} + 2<\partial_{x}\phi,\sqrt{2}\Pi_{\mathbb{P}_N}\partial_{x}^{*}\tilde{C}_{2} -\partial_{x}\tilde{C}_{0}> \phi_{s}\\
   & = & \Pi_{\mathbb{P}_N}\partial_{x}^{*}(\sqrt{2}\Pi_{\mathbb{P}_N}\partial_{x}^{*}\tilde{C}_{2} -\partial_{x}\tilde{C}_{0}) -(\sqrt{2}<\partial_{x}^{3}\phi,\tilde{C}_{2}>  - <\partial_{x}^{*}\partial_{x}^{2}\phi,\tilde{C}_{0}> )x\\
&  - &  \dfrac{1}{2}  (\sqrt{2}<\partial_{x}^{2}\Pi_{\mathbb{P}_N}\partial_{x}^{*2}1,\tilde{C}_{2}>-<\partial_{x}^{*}\partial_{x}\Pi_{\mathbb{P}_N}\partial_{x}^{*2}1,\tilde{C}_{0}>) \xi_{2}\\
& + & 2(\sqrt{2}<\partial_{x}^{2}\phi,\tilde{C}_{2}> -<\partial_{x}^{*}\partial_{x}\phi, \tilde{C}_{0}>) \phi_{s}.\\
\end{eqnarray*}

\end{proof}

\subsection{Control of the multidimensionnal quantities}

In this subsection, we provide estimates on the time derivatives of some terms  which will appear in the entropy functionnal. These terms allow us to recover dissipation for the multidimensionnal quantities $\tilde{e}_s,\tilde{m}_s$ and $\tilde{w}_s$.

\begin{lemme}
There exists two positive constants $\kappa_{1}$ and $C_{N}$ such that 

$$
\dfrac{d}{dt} <\Omega^{-1}\partial_{x}\tilde{C}_{2},\tilde{C}_{3}> \leq -\kappa_{1} \|e_{s}\|^{2} + C_{N} \|\tilde{h} \|_{L^{2}(\mathcal{M})} \|\tilde{h}^{\perp} \|_{L^{2}(\mathcal{M})}.
$$
\label{lemme/1}
\end{lemme}

\begin{proof}
We use the semi-discrete scheme (\ref{equation/schema_discret_2})  to compute the time derivative explicitly:

\begin{eqnarray}
\dfrac{d}{dt} <\Omega^{-1}\partial_{x}\tilde{C}_{2},\tilde{C}_{3}>  & = & -\sqrt{3} \| \Omega^{-1/2}\partial_{x}\tilde{C}_{2}\|^{2} \nonumber\\
 & + & <\Omega^{-1}\partial_{x}(\sqrt{3}\Pi_{\mathbb{P}_N}\partial_{x}^{*}\tilde{C}_{3}-\sqrt{2}\partial_{x}\tilde{C}_{1}),\tilde{C}_{3}> \nonumber \\
 & + &  <\Omega^{-1}\partial_{x}\tilde{C}_{2},2\Pi_{\mathbb{P}_N}\partial_{x}^{*}\tilde{C}_{4}-\tilde{C}_{3}>. \label{equation/20}
\end{eqnarray}

 The second term in the right-hand side of Equation (\ref{equation/20}) can be bounded by using the Cauchy-Schwarz inequality, along with  (\ref{equation/poincare1}) and (\ref{equation/H1}):

$$
|<\Omega^{-1}\partial_{x}(\sqrt{3}\Pi_{\mathbb{P}_N}\partial_{x}^{*}\tilde{C}_{3}-\sqrt{2}\partial_{x}\tilde{C}_{1}),\tilde{C}_{3}>| \lesssim (\sqrt{3}K_{N} + \sqrt{2})\|\tilde{h} \|_{L^{2}(\mathcal{M})} \|\tilde{h}^{\perp} \|_{L^{2}(\mathcal{M})}.
$$

The last term in the right-hand side of Equation (\ref{equation/20}) can be estimated, by using the Young inequality,  inequality (\ref{equation/H0}) and the continuity of $\Omega^{-1/2}$: 

\begin{eqnarray*}
 <\Omega^{-1}\partial_{x}\tilde{C}_{2},2\Pi_{\mathbb{P}_N}\partial_{x}^{*}\tilde{C}_{4}-\tilde{C}_{3}> & \leq & \dfrac{\sqrt{3}}{2} \| \Omega^{-1/2}\partial_{x}\tilde{C}_{2}\|^{2} + \dfrac{1}{2\sqrt{3}} \| \Omega^{-1/2}(2\Pi_{\mathbb{P}_N}\partial_{x}^{*}\tilde{C}_{4}-\tilde{C}_{3})\|^{2} \\
 & \leq & \dfrac{\sqrt{3}}{2} \| \Omega^{-1/2}\partial_{x}\tilde{C}_{2}\|^{2} + \dfrac{4}{\sqrt{3}}  K_{N}^{2}\|\tilde{C}_{4}\|^{2}  + \dfrac{1}{\sqrt{3}} \|\tilde{C}_{3}\|^{2}.
\end{eqnarray*}

We finally gather all these estimations and use Equation (\ref{equation/20}) to get that

\begin{eqnarray*}
\dfrac{d}{dt} <\Omega^{-1}\partial_{x}\tilde{C}_{2},\tilde{C}_{3}>  & \leq & -\dfrac{\sqrt{3}}{2} \| \Omega^{-1/2}\partial_{x}\tilde{C}_{2}\|^{2} + C_{N} \|\tilde{h} \|_{L^{2}(\mathcal{M})} \|\tilde{h}^{\perp} \|_{L^{2}(\mathcal{M})}. \\
\end{eqnarray*}

We conclude by using inequality (\ref{equation/poincare2}) and the expression of $\tilde{e}_s$ in (\ref{equation/15}).

\end{proof}

\begin{lemme}
There exists two positive constants  $\kappa_{2}$ and $C_{N}$ such that 

$$
\dfrac{d}{dt} <\Omega^{-1}\partial_{x}m_{s},e_{s}> \leq -\kappa_{2} \|m_{s}\|^{2} + C_{N}( \|e_{s}\| +  \|\tilde{h}^{\perp} \|_{L^{2}(\mathcal{M})}) \| \tilde{h} \|_{L^{2}(\mathcal{M})}.
$$
\label{lemme/2}
\end{lemme}

\begin{proof}

We use equation (\ref{equation/a}) in order to compute the time derivative explicitly:

\begin{eqnarray}
\dfrac{d}{dt} <\Omega^{-1}\partial_{x}m_{s},e_{s}>  & = & -\sqrt{2} \| \Omega^{-1/2}\partial_{x}m_{s}\|^{2} \nonumber  \\
 & + &  \sqrt{3}<\Omega^{-1}\partial_{x}m_{s},\Pi_{\mathbb{P}_N}\partial_{x}^{*}\tilde{C}_{3}> \nonumber \\
 & + &  <\Omega^{-1}\partial_{x}\partial_{t}m_{s},e_{s}>. \label{equation/22}
\end{eqnarray}

The second term in the right-hand side of Equation (\ref{equation/22}) can be bounded by using Young inequality along with (\ref{equation/H0}):

$$
\sqrt{3}|<\Omega^{-1}\partial_{x}m_{s},\Pi_{\mathbb{P}_N}\partial_{x}^{*}\tilde{C}_{3}>| \leq \dfrac{\sqrt{2}}{2} \| \Omega^{-1/2}\partial_{x}m_{s}\|^{2} + \dfrac{\sqrt{3}}{2\sqrt{2}} K_{N}^{2} \|\tilde{h}^{\perp} \|_{L^{2}(\mathcal{M})}^{2}.
$$

The  last  term in the right-hand side of Equation (\ref{equation/22}) is bounded by using Cauchy-Schwarz inequality:

$$
 <\Omega^{-1}\partial_{x}\partial_{t}m_{s},e_{s}> \leq \| e_{s}\|  \| \Omega^{-1}\partial_{x}\partial_{t}m_{s} \|.
$$
Remark that $\Omega 1 = 1$, so $\Omega^{-1}1 = 1$. We compute explicitely  $\Omega^{-1}\partial_{x}\partial_{t}m_{s}$ with the help of equation (\ref{equation/b}):

$$
\Omega^{-1}\partial_{x}\partial_{t}m_{s} = \sqrt{2} \Omega^{-1}\partial_{x}\Pi_{\mathbb{P}_N}\partial_{x}^{*}\tilde{C}_{2} -\Omega^{-1}\partial_{x}^{2}\tilde{C_{0}} - \sqrt{2} <\partial_{x}\Pi_{\mathbb{P}_N}\partial_{x}^{*}\tilde{C}_{2}> + <\partial_{x}^{2}\tilde{C}_{0}>.
$$

By some integration by parts in the last two terms and since $\partial_{x}\phi\in \mathbb{P}_N$, we get that

$$
\Omega^{-1}\partial_{x}\partial_{t}m_{s} = \sqrt{2} \Omega^{-1}\partial_{x}\Pi_{\mathbb{P}_N}\partial_{x}^{*}\tilde{C}_{2} -\Omega^{-1}\partial_{x}^{2}\tilde{C_{0}} - \sqrt{2} <\tilde{C}_{2},\partial_{x}^{2}\phi> + <\tilde{C}_{0}, \partial_{x}^{*}\phi>.
$$

By  Inequalities (\ref{equation/poincare1}), (\ref{equation/H1}), we find that

$$
\| \Omega^{-1}\partial_{x}\partial_{t}m_{s} \| \lesssim C_{N} \|\tilde{h} \|_{L^{2}(\mathcal{M})} + \sqrt{2} \| \partial_{x}^{2}\phi \| \|\tilde{C}_{2} \| + \| \tilde{C}_{0} \| \| \partial_{x}^{*}\phi \| \lesssim C_{N} \|\tilde{h} \|_{L^{2}(\mathcal{M})}.
$$

We conclude the proof by gathering the estimations:

$$
\dfrac{d}{dt} <\Omega^{-1}\partial_{x}m_{s},e_{s}>   \leq  -\dfrac{\sqrt{2}}{2} \| \Omega^{-1/2}\partial_{x}m_{s}\|^{2} + C_{N} (\|e_{s}\| + \|\tilde{h}^{\perp}\|) \|\tilde{h} \|_{L^{2}(\mathcal{M})} \\
$$

and then by applying the Poincaré-Lions inequality (\ref{equation/poincare2}).
\end{proof}

\begin{lemme}
There exists two positive constants $\kappa_{3}$ and $C_{N}$ such that 

$$
\dfrac{d}{dt} <\Omega^{-1}\partial_{x}w_{s},m_{s}> \leq -\kappa_{3} \|w_{s}\|^{2} + C_{N}( \|e_{s}\|^{2} +  \|\tilde{h}^{\perp} \|_{L^{2}(\mathcal{M})}^{2} + \|m_{s}\| \|\tilde{h} \|_{L^{2}(\mathcal{M})}).
$$

and

$$
\dfrac{d}{dt} <-\Omega^{-1}\partial_{t}w_{s},w_{s}> \leq -\|\Omega^{-1/2}\partial_{t}w_{s}\|^{2} + C_{N} \|w_{s}\| \|\tilde{h} \|_{L^{2}(\mathcal{M})}.
$$
\label{lemme/3}
\end{lemme}

\begin{proof}

By  Equation (\ref{equation/c}) and by using Inequality (\ref{equation/H1}) we have that

\begin{eqnarray}
<\Omega^{-1}\partial_{x}w_{s},\partial_{t}m_{s}> & = & <\Omega^{-1}\partial_{x}w_{s},-\partial_{x}w_{s} + \sqrt{2}(\Pi_{\mathbb{P}_N}\partial_{x}^{*}e_{s}-  <e_{s}\partial_{x}^{2}\phi> x)> \nonumber\\
 & = & -\|\Omega^{-1/2}\partial_{x}w_{s}\|^{2} +  \sqrt{2}<\Omega^{-1/2}\partial_{x}w_{s},\Omega^{-1/2}(\Pi_{\mathbb{P}_N}\partial_{x}^{*}e_{s}-  <e_{s}\partial_{x}^{2}\phi> x)> \nonumber \\
 & \leq &  -\dfrac{1}{2}\|\Omega^{-1/2}\partial_{x}w_{s}\|^{2} + \|\Omega^{-1/2}(\Pi_{\mathbb{P}_N}\partial_{x}^{*}e_{s}-  <e_{s}\partial_{x}^{2}\phi> x)\|^{2} \nonumber \mbox{ (Young)}   \\
  & \leq &  -\dfrac{1}{2}\|\Omega^{-1/2}\partial_{x}w_{s}\|^{2} + 2 \|\Omega^{-1/2}\Pi_{\mathbb{P}_N}\partial_{x}^{*}e_{s}\|^{2} + 2 \|\Omega^{-1/2}  <e_{s}\partial_{x}^{2}\phi> x\|^{2} \nonumber \\
    & \leq &  -\dfrac{1}{2}\|\Omega^{-1/2}\partial_{x}w_{s}\|^{2} + 2 K_{N}^{2} \|e_{s}\|^{2} + 2 \|\Omega^{-1/2} x\|^{2} \| \partial_{x}^{2}\phi \|^{2} \|e_{s}\|^{2}.     \label{equation/23}
\end{eqnarray}

We continue by computing $\Omega^{-1}\partial_{x}\partial_{t}w_{s}$ with the help of equation (\ref{equation/d}):
$$
\Omega^{-1}\partial_{x}\partial_{t}w_{s} = \Omega^{-1} \partial_{x}\Pi_{\mathbb{P}_N}\partial_{x}^{*}\tilde{C}_{1} -< \partial_{x}^{2}\phi, \tilde{C}_{1}> - <\partial_{x}\Pi_{\mathbb{P}_N}\partial_{x}^{*2}1,\tilde{C}_{1}> \Omega^{-1}x + 2<\partial_{x}\phi,\tilde{C}_{1}> \Omega^{-1}\partial_{x}\phi_{s}.
$$

By Inequality (\ref{equation/H1}) and the Cauchy-Schwarz inequality, we get

\begin{equation}
<\Omega^{-1}\partial_{x}\partial_{t}w_{s},m_{s}>  \leq   \|m_{s}\| \|\tilde{h} \|_{L^{2}(\mathcal{M})} (K_{N} + \|\partial_{x}^{2}\phi \| +\|\Omega^{-1}x\|  \| \partial_{x}\Pi_{\mathbb{P}_N} \partial_{x}^{*2}1\| + 2\|\partial_{x}\phi\| \|\Omega^{-1} \partial_{x}\phi_{s}\| ).
\label{equation/24}
\end{equation}

We conclude by (\ref{equation/23}), (\ref{equation/24}) that 

$$
\dfrac{d}{dt}<\Omega^{-1}\partial_{x}w_{s},m_{s}> \leq -\dfrac{1}{2} \|\Omega^{-1/2}\partial_{x}w_{s}\|^{2} + C_{N} \|e_{s}\|^{2} + C_{N} \|\tilde{h} \|_{L^{2}(\mathcal{M})} \|m_{s}\|.
$$

Let us move on to the second part of  Lemma \ref{lemme/3}. We have 

$$
\dfrac{d}{dt}<-\Omega^{-1}\partial_{t}w_{s}, w_{s}> = -\|\Omega^{-1/2}\partial_{t}w_{s}\|^{2} - <\Omega^{-1}\partial_{t}^{2}w_{s},w_{s} >.
$$

Moreover, $\Omega^{-1}\partial_{t}^{2}w_{s}$ is known explicitely by (\ref{equation/e}):

\begin{eqnarray*}
\Omega^{-1} \partial_{t}^{2}w_{s} & = & \sqrt{2} \Omega^{-1}\Pi_{\mathbb{P}_N}\partial_{x}^{*}\Pi_{\mathbb{P}_N}\partial_{x}^{*}\tilde{C}_{2} \\
 &- &\Omega^{-1}\Pi_{\mathbb{P}_N}\partial_{x}^{*} \partial_{x}\tilde{C}_{0} \\
 & - & (\sqrt{2}<\partial_{x}^{3}\phi,\tilde{C}_{2}>  - <\partial_{x}^{*}\partial_{x}^{2}\phi,\tilde{C}_{0}> )\Omega^{-1} x \\
&  - &  \dfrac{1}{2}  (\sqrt{2}<\partial_{x}^{2}\Pi_{\mathbb{P}_N}\partial_{x}^{*2}1,\tilde{C}_{2}>-<\partial_{x}^{*}\partial_{x}\Pi_{\mathbb{P}_N}\partial_{x}^{*2}1,\tilde{C}_{0}>) \Omega^{-1}\xi_{2} \\
& + & 2(\sqrt{2}<\partial_{x}^{2}\phi,\tilde{C}_{2}> -<\partial_{x}^{*}\partial_{x}\phi, \tilde{C}_{0}>) \Omega^{-1}\phi_{s}.
\end{eqnarray*}

The first two lines are bounded by (\ref{equation/H2}) and (\ref{equation/H3}). The last three lines can be bounded by using the Cauchy-Schwarz inequality. We find that

$$
\| \Omega^{-1}\partial_{t}^{2}w_{s} \| \leq C_{N} \|\tilde{h} \|_{L^{2}(\mathcal{M})}.
$$

Hence,

$$
\dfrac{d}{dt}<-\Omega^{-1}\partial_{t}w_{s}, w{s}> = -\|\Omega^{-1/2}\partial_{t}w_{s}\|^{2} +  C_{N} \|\tilde{h} \|_{L^{2}(\mathcal{M})}^2.
$$
\end{proof}

\subsection{Control of the one dimensional quantities}

Let us define the three following quantities. They will appear in the entropy functional so our goal is to control their time derivatives.

\begin{equation}
b(t) := <\tilde{C}_{1}>
\label{equation/25}
\end{equation}

\begin{equation}
c(t) := <\tilde{C}_{2}>
\label{equation/26}
\end{equation}

\begin{equation}
z(t,x) := \tilde{r}(t,x) + b'(t)x - c''(t)\dfrac{\xi_{2}(x)}{2\sqrt{2}}- \sqrt{2}c(t) b\xi_{\phi}(x)
\label{equation/27}
\end{equation}

Remember that we have already defined

$$
\tilde{r}(t,x) = \tilde{C}_{0}(t,x); \ \tilde{m}(t,x) = \tilde{C}_{1}(t,x); \ \tilde{e}(t,x) = \tilde{C}_{2}(t,x) 
$$

We start with a short lemma.

\begin{lemme}
The functions $b,c,z$ satisifies the following relations:

$$
\tilde{r}(t,x) = -b'(t)x + c''(t)\dfrac{\xi_{2}(x)}{2\sqrt{2}}+ \sqrt{2} c(t) + z(t,x),
$$
$$
\tilde{e}(t,x) = c(t)+e_{s}(t,x),
$$
$$
\tilde{m}(t,x) = b(t)-\dfrac{1}{\sqrt{2}}c'(t)x + m_{s}(t,x).
$$

\label{lemme/4}
\end{lemme}

\begin{proof}
Only the last relation is not trivial and does not follow directly from the definitions.

\noindent For the last relation, notice that: 
$$
c'(t) =<\partial_{t}\tilde{C}_{2}> = - \sqrt{2}<\partial_{x}\tilde{C}_{1}> .
$$

\noindent By the definition of $m_{s}$, we get the result.

\end{proof}

\begin{lemme}
The function $z$ satisfies :

$$
\| z \|^{2} \lesssim \| w_{s} \|^{2} + \| e_{s} \|^{2} ,
$$

$$
\|\Omega^{-1/2}\partial_{t}z \|^{2} \lesssim \|\Omega^{-1/2}\partial_{t}w_{s} \|^{2} +\| m_{s} \|^{2} + \|\tilde{h}^{\perp} \|_{L^{2}(\mathcal{M})}^{2}.
$$

\label{lemme/5}
\end{lemme}

\begin{proof}
From  the definiton of $w_{s}$ (\ref{equation/18}), we deduce that

\begin{equation}
z := r+ b'x - c'' \dfrac{\xi_{2}}{2\sqrt{2}}-\sqrt{2}c \xi_{\phi} = w_{s} + \left(-\dfrac{1}{\sqrt{2}}c <\partial_{x}^{2}\phi>+\dfrac{1}{2}<\partial_{x}^{2}\tilde{C}_{0}>-\dfrac{c''}{2\sqrt{2}}\right)\xi_{2}.
\label{equation/-1}
\end{equation}

Then, we compute  $c''$ by using  System (\ref{equation/schema_discret_2}) and the fact that $\partial_{x}\phi\in \mathbb{P}_N$:

\begin{eqnarray*}
c''(t) & = & \sqrt{3} <\Pi_{\mathbb{P}_N}\partial_{x}^{*}\partial_{t}\tilde{C}_{3}> - \sqrt{2}  <\partial_{x}\partial_{t}\tilde{C}_{1}> \\
&=& - \sqrt{2}  <\partial_{x}\partial_{t}\tilde{C}_{1}> \\
 & = & - \sqrt{2}  <\partial_{x}(\sqrt{2}\Pi_{\mathbb{P}_N}\partial_{x}^{*}\tilde{C}_{2}-\partial_{x}\tilde{C}_{0})> \\
 & = & - 2  <\tilde{C}_{2}\partial_{x}^{2}\phi>+\sqrt{2}<\partial_{x}^{2}\tilde{C}_{0}>. \\ 
\end{eqnarray*}

This expression for $c''$ can be injected in the expression (\ref{equation/-1}), and we get

$$
z = w_{s} + \dfrac{1}{\sqrt{2}} <e_{s}\partial_{x}^{2}\phi>\xi_{2}.
$$

The Cauchy-Schwarz and triangle inequalities implies that

$$
\| z \| \leq \|w_{s}\| + \dfrac{1}{\sqrt{2}} \|\xi_{2}\| \|\partial_{x}^{2}\phi\| \|e_{s}\|.
$$

\noindent This last expression shows the first point. For the second point, we differentiate $z$ with respect to $t$:

$$
\partial_{t}z = \partial_{t}w_{s} + \dfrac{1}{\sqrt{2}} <\partial_{t}e_{s}\partial_{x}^{2}\phi>\xi_{2}.
$$

\noindent We then use the scheme (\ref{equation/schema_discret_2}) and the definition of $m_{s}$ (\ref{equation/14}) to prove that:

$$
\partial_{x}\tilde{C}_{1} = \partial_{x}m_{s} + \dfrac{1}{\sqrt{2}}<\partial_{t}\tilde{C}_{2}>.
$$

By using (\ref{equation/schema_discret_2}) again:

$$
\sqrt{2} \partial_{t}\tilde{C}_{2} = \sqrt{6} \Pi_{\mathbb{P}_N}\partial_{x}^{*}\tilde{C}_{3}- 2\partial_{x}m_{s} + \sqrt{2}<\partial_{t}\tilde{C}_{2}>.
$$

This implies that

$$
\sqrt{2}e_{s} = \sqrt{6} \Pi_{\mathbb{P}_N}\partial_{x}^{*}\tilde{C}_{3}- 2\partial_{x}m_{s} .
$$

We have computed  $\partial_{t}e_{s}$, so we can replace it in $\partial_{t}z$:

$$
\partial_{t}z = \partial_{t}w_{s} + \dfrac{1}{2} <\sqrt{6} \Pi_{\mathbb{P}_N}\partial_{x}^{*}\tilde{C}_{3}- 2\partial_{x}m_{s} ,\partial_{x}^{2}\phi>\xi_{2}.
$$

By using that  $\partial_{x}\phi\in \mathbb{P}_N$ and the Cauchy-Schwarz and triangle inequalities, we prove that 

$$
\|\Omega^{-1/2}\partial_{t}z\| \leq \| \Omega^{-1/2}\partial_{t}w_{s}\| +\left( \dfrac{\sqrt{6}}{2} \|\partial_{x}^{3}\phi\| \|\tilde{C}_{3}\| + \|m_{s}\| \|\partial_{x}^{*}\partial_{x}^{2}\phi \| \right) \|\Omega^{-1/2}\xi_{2}\|.
$$

This gives us the last point.

\end{proof}

The following lemma gives a master equation satisfied by $z,b$ and $c$. This equation is slightly different from Equation (4.31) in \cite{macroscopicmode}, because of the orthogonal projection $\Pi_{\mathbb{P}_N}$. Under the hypothesis \textbf{(H)}, it will not raise any issue.

\begin{lemme}
The functions $m_{s},b,c,z$ satisfy:

\begin{equation}
\dfrac{2\xi_{\phi}+x\partial_{x}\phi-1}{\sqrt{2}}c' + \dfrac{\xi_{2}}{2\sqrt{2}}c''' - \partial_{x}\phi b- xb'' = \Pi_{\mathbb{P}_N}\partial_{x}^{*}m_{s} - \partial_{t}z .
\label{equation/28}
\end{equation}

This motivates the definition $R_{0}  :=  \Pi_{\mathbb{P}_N}\partial_{x}^{*}m_{s} - \partial_{t}z $.
\end{lemme}

\begin{proof}

From the definition of $z$,  

$$
\partial_{t} \tilde{C}_{0} = \sqrt{2} c' \xi_{\phi} - b'' x  + \dfrac{\xi_{2}}{2\sqrt{2}} c''' + \partial_{t}z.
$$

By the scheme (\ref{equation/schema_discret_2}), the third part of Lemma  (\ref{lemme/4}) and since $x\partial_{x}\phi\in \mathbb{P}_N$:

$$
\partial_{t} \tilde{C}_{0}  = \Pi_{\mathbb{P}_N} \partial_{x}^{*}\tilde{C}_{1} = b\partial_{x}\phi -\dfrac{1}{\sqrt{2}} c' (x\partial_{x}\phi - 1)+ \Pi_{\mathbb{P}_N}\partial_{x}^{*}m_{s}.
$$

As the two expressions are equal, the result follows.

\end{proof}

We will now exploit the ODE (\ref{equation/28}) to find estimates on the derivatives of $b,c$ and $z$.

\begin{lemme}
The functions  $b$ and $c$ verify 

$$
|b| + |b''| + |c'| + |c'''| \lesssim \|\Omega^{-1/2}\partial_{t}w_{s}\| + \|m_{s}\| + \|\tilde{h}^{\perp}\|.
$$

\label{lemme/7} 
\end{lemme}

\begin{proof}
We start by multiplying  Equation (\ref{equation/23}) by $\partial_{x}\rho$ and then integrating.
As $<1>=<(\partial_{x}\phi)^{2}>=1, <x>=0$ we find that 

$$
-\dfrac{1}{\sqrt{2}}<x(\partial_{x}\phi)^{2}>c' + b + b'' = -<\partial_{x}\phi R_{0}>.
$$

By injecting this last expression for $b''$ into (\ref{equation/28}) we obtain that 

\begin{equation}
\left(\dfrac{2\xi_{\phi}+x\partial_{x}\phi-1}{\sqrt{2}}-x \dfrac{<x(\partial_{x}\phi)^{2}>}{\sqrt{2}}\right) c' + \dfrac{\xi_{2}}{2\sqrt{2}}c''' - (\partial_{x}\phi - x)b = R_{0} - x <\partial_{x}\phi R_{0}>.
\label{equation/29}
\end{equation}

\noindent By multiplying by $ (\partial_{x}\phi -x)\rho$ and then integrating, we find that

$$
\alpha_{1} c' + \alpha_{2} c''' - <(\partial_{x}\phi - x)^{2}>b = <(\partial_{x}\phi - x-<x(\partial_{x}\phi - x)> \partial_{x}\phi) R_{0}>.
$$

\noindent The real numbers $\alpha_{1},\alpha_{2}$ are given by 

$$
\alpha_{1} = <\dfrac{2\xi_{\phi}+x\partial_{x}\phi-1}{\sqrt{2}}-x \dfrac{<x(\partial_{x}\phi)^{2}>}{\sqrt{2}},\partial_{x}\phi -x >,
$$

$$
\alpha_{2} = <\dfrac{\xi_{2}}{2\sqrt{2}},\partial_{x}\phi -x >.
$$

\noindent We then set $\tilde{\Phi} := \partial_{x}\phi - x-<x(\partial_{x}\phi - x)> \partial_{x}\phi$ which is an element of $\mathbb{P}_N$. By definition of $R_{0}$, we have

 $$
 <R_{0}\tilde{\Phi}> = <\Pi_{\mathbb{P}_N}\partial_{x}^{*}m_{s}\tilde{\Phi}> - <\partial_{t}z,\partial_{t} \tilde{\Phi}> = <m_{s}, \partial_{x}\tilde{\Phi}> - <\Omega^{-1/2}\partial_{t}, \Omega^{1/2}\tilde{\Phi}>. $$

Functions $b$ and $c$ satisfy the following differential equation:

$$
 <(\partial_{x}\phi - x)^{2}> b = \alpha_{1} c' + \alpha_{2} c''' - <m_{s},\partial_{x}\tilde{\Phi}> + <\Omega^{-1/2}\partial_{t}, \Omega^{1/2}\tilde{\Phi}>.
$$

There are then two cases.

\paragraph{\textbf{Quadratic case:}} the potential $\phi$ is $\phi(x)=\dfrac{1}{2}(x^{2}+\ln(2\pi))$. We know from studying the conservation laws that $b=c=0$, which gives the result directly.

\paragraph{\textbf{Non-quadratic case:}} $M_{\Phi} := <(\partial_{x}\phi - x)^{2}>\neq 0$ and so $b$ verifies 

$$
   b = \dfrac{\alpha_{1}}{M_{\Phi}} c' + \dfrac{\alpha_{2}}{M_{\Phi}} c''' - \dfrac{<m_{s},\partial_{x}\tilde{\Phi}>}{M_{\Phi}} + \dfrac{<\Omega^{-1/2}\partial_{t}, \Omega^{1/2}\tilde{\Phi}>}{M_{\Phi}}
$$

and we can inject this expression into (\ref{equation/24}), which gives that 

\begin{small}
\begin{equation}
\Psi_{1}(x) c' + \Psi_{2}(x) c''' = R_{0} - x <\partial_{x}\phi R_{0}> - \dfrac{\partial_{x}\phi - x}{M_{\Phi}} <\partial_{x}\tilde{\Phi},m_{s}>  + \dfrac{\partial_{x}\phi - x}{M_{\Phi}} <\Omega^{-1/2}\partial_{t}z, \Omega^{1/2}\tilde{\Phi}>.
\label{equation/30}
\end{equation}
\end{small}

The functions $\Psi_{1},\Psi_{2}$ are given by

$$
\Psi_{1}(x) := \dfrac{2\xi_{\phi}+x\partial_{x}\phi-1}{\sqrt{2}}-x \dfrac{<x(\partial_{x}\phi)^{2}>}{\sqrt{2}} - \alpha_{1}  \dfrac{\partial_{x}\phi - x}{M_{\Phi}} ,  
$$

$$
\Psi_{2}(x) := \dfrac{\xi_{2}}{2\sqrt{2}} -\dfrac{\alpha_{2}}{M_{\Phi}} (\partial_{x}\phi-x).
$$

Now, assume that the functions $\Psi_{1}$ and $\Psi_{2}$ are linearly dependant, so there exists $\lambda\neq 0$ such that  $\Psi_{1} = \lambda\Psi_{2}$. It means that there exists real numbers $\alpha,\beta,\gamma$ such that for all $x\in\mathbb{R}$:

$$
2\phi(x) + x\partial_{x}\phi(x) + \alpha \partial_{x}\phi(x) = \dfrac{\lambda}{2}x^{2} + \beta x + \gamma .
$$

As $\phi$ is not harmonic, $deg(\phi)=2m\geq 4$. Let $\gamma_m$ be the positive leading coefficient of $\phi$. The equation above says that the term with highest degree in the left hand side vanishes, so we deduce that 

$$
2\gamma_{m}(1+m)=0.
$$

As $\gamma_{m}\neq 0 $, we get that $m=-1$ which is impossible. It shows that the functions are linearly independant. As  $\Psi_{1}$ and $\Psi_{2}$ are lineraly independent and  belong to $\mathbb{P}_N$, we can find a basis $(\Psi_{1},\Psi_{2},w_{3},...w_{N+1})$ of $\mathbb{P}_N$. Then, define $P : \mathbb{P}_N \mapsto \mathbb{P}_N$ as the non-symmetric projector of range $Span(\Psi_{1})$ and of kernel $Span(\Psi_{2},w_{3},...w_{N+1})$. By applying $P$ to (\ref{equation/30}), we get that

$$
\Psi_{1}(x) c' = P(R_{0}) - P(x) <\partial_{x}\phi R_{0}> - \dfrac{P(\partial_{x}\phi - x)}{M_{\Phi}} <\partial_{x}\tilde{\Phi},m_{s}>  + \dfrac{P(\partial_{x}\phi - x)}{M_{\Phi}} <\Omega^{-1/2}\partial_{t}z, \Omega^{1/2}\tilde{\Phi}>.
\label{equation/32}
$$

Now, take the scalar product with $\Psi_{1}$ to get that 

\begin{eqnarray*}
\|\Psi_{1}\|^{2} c'&=& <P(R_{0}),\Psi_{1}> - <P(x),\Psi_{1}> <\partial_{x}\phi R_{0}> \\
& + & \dfrac{<P(\partial_{x}\phi - x),\Psi_{1}>}{M_{\Phi}} (- <\partial_{x}\tilde{\Phi},m_{s}>  + <\Omega^{-1/2}\partial_{t}z, \Omega^{1/2}\tilde{\Phi}>).
\end{eqnarray*}

It only remains to estimate the right-hand side. Remark that the second line is directly controled by the Cauchy-Schwarz inequality and Lemma \ref{lemme/5}. We focus now on the first line. Since the adjoint $P^{*}$ of $P$ sends $\mathbb{P}_N$ into $\mathbb{P}_N$,  the first term can be written as

$$
<m_{s},\partial_{x}P^{*}(\Psi_{1})> - <\Omega^{-1/2}\partial_{t}z,\Omega^{1/2}P^{*}(\Psi_{1})>.
$$

This term can be controlled by using Lemma \ref{lemme/5}. Finally,  the second term of the first line can be written as 

$$
 <P(x),\Psi_{1}> (<\partial_{x}^{2}\phi,m_{s}> - <\Omega^{1/2}\partial_{x}\phi, \Omega^{-1/2}\partial_{t}z>)
$$

\noindent and it can be controlled by Lemma \ref{lemme/5}. The computations are the same when replacing $P$ by the projector of range $Span(\Psi_{2})$ and of kernel $Span(\Psi_{1},w_{3},...w_{N+1})$. We can use the estimates on $c',c'''$ in the expressions of $b, b''$ found above and use Lemma \ref{lemme/5} to finish the proof.

\end{proof}

\begin{lemme}
There exists a constant $C>0$ such that 

$$
\dfrac{d}{dt}(-bb') \leq -b^{'2} + C ( \|\Omega^{-1/2}\partial_{t}w_{s}\|^{2} + \|m_{s}\|^{2} + \|\tilde{h}^{\perp}\|^{2}),
$$

$$
\dfrac{d}{dt}(-c'c'') \leq -c^{''2} + C ( \|\Omega^{-1/2}\partial_{t}w_{s}\|^{2} + \|m_{s}\|^{2} + \|\tilde{h}^{\perp}\|^{2}).
$$
\end{lemme}

\begin{proof}
By direct calculation,

$$
\dfrac{d}{dt}(-bb') \leq -b^{'2} - bb'',
$$

$$
\dfrac{d}{dt}(-c'c'') \leq -c^{''2} - c'c'''.
$$

The result follows from Lemma \ref{lemme/7}.

\end{proof}

\begin{lemme}
We have that 

$$
|c| \lesssim |b'| + |c''| + \|w_{s}\| + \|e_{s}\| ,
$$

$$
\|\tilde{C}_{0} \| \lesssim |b'| + |c''| + \|w_{s}\| + \|e_{s}\| .
$$
\label{lemme/9}
\end{lemme}

\begin{proof}
According to  Lemma \ref{lemme/4},

$$
\tilde{C}_{0} = \sqrt{2} \xi_{\phi} c + S
$$

where $S:= z - b'x + \dfrac{\xi_{2}}{2\sqrt{2}}c''$. By  Lemma \ref{lemme/5}, $\|S\| \lesssim |b'|+|c''| + \|w_{s}\| + \|e_{s}\| $.

\noindent Remember the conservation law (see Proposition \ref{proposition/2}):

$$
\dfrac{1}{\sqrt{2}}c+ <\phi \tilde{C}_{0}> = 0.
$$

\noindent Replacing $\tilde{C}_{0}$, we have that 

$$
\dfrac{1}{\sqrt{2}}c (1+2<\phi\xi_{\phi}>) = -<\phi S>.
$$

\noindent Since $<\phi \xi_{\phi}> = <\xi_{\phi}^{2}>$, we obtain the following estimate  which implies the first line of the lemma:

$$
|c| \leq \dfrac{\sqrt{2}}{1+<\xi_{\phi}^{2}>} \|\phi| \|S\|
$$

Returning to the expression $\tilde{C}_{0} = \sqrt{2} \xi_{\phi} c + S$ and  injecting the estimate for $c$, we obtain the second line of the lemma.
\end{proof}

\subsection{Entropy and proof of hypocoercivity}

We start by defining an entropy 

\begin{eqnarray}
\mathcal{H}_{1}(t) &= &\|\tilde{h} \|_{L^{2}(\mathcal{M})}^{2} + \varepsilon <\Omega^{-1}\partial_{x}\tilde{C}_{2},\tilde{C}_{3}> + \varepsilon^{\frac{3}{2}} <\Omega^{-1}\partial_{x}m_{s}, e_{s}> + \varepsilon^{\frac{7}{4}} <\Omega^{-1}\partial_{x}w_{s},m_{s}> \nonumber \\ 
 &+ & \varepsilon^{\frac{15}{8}} <-\Omega^{-1}\partial_{t}w_{s},w_{s}> 
\end{eqnarray}

\noindent where $\varepsilon>0$ is a parameter to be chosen later on, and the associated dissipation

\begin{equation}
\mathcal{D}_1(t) = \|\tilde{h}^{\perp}\|^{2} + \|e_{s}\|^{2} + \|m_{s}\|^{2} + \|w_{s}\|^{2} + \|\Omega^{-1/2}\partial_{t}w_{s}\|^{2}.
\end{equation}

Note that Lemmas \ref{lemme/1},\ref{lemme/2} and \ref{lemme/3} are the equivalents of  Lemmas 4.3, 4.4 and 4.5 of \cite{macroscopicmode}. The only difference here is that the constant $C_{N}$ depends on $N$, the truncation parameter. Thus, we can state the following lemma, whose proof is the same as Lemma 4.6 of \cite{macroscopicmode}. For the sake of completeness, we will also give the proof.

\begin{proposition}
There are positive constants:

\begin{itemize}
\item $\kappa,\kappa_{0}$ independent of $N$,
\item $C_{0,N}$ dependent on $N$,
\item $\varepsilon_{N}$ dependent on $N$ and small enough 
\end{itemize}

\noindent such that

$$
\dfrac{d}{dt}\mathcal{H}_{1}(t) \leq -\kappa_{0} \|\tilde{h}^{\perp} \|_{L^{2}(\mathcal{M})}^{2} - \varepsilon^{15/8}_{N} \kappa \mathcal{D}_{1}(t) + \varepsilon_{N}^{2} C_{0,N} \|\tilde{h} \|_{L^{2}(\mathcal{M})}^{2}.
$$
\label{proposition/4}
\end{proposition}

\begin{proof}

We can estimate the time derivative of the entropy by using Lemmas \ref{lemme/1},\ref{lemme/2}, \ref{lemme/3}:

\begin{eqnarray*}
\dfrac{d}{dt}\mathcal{H}_{1}(t) & \leq & -2 \|\tilde{h}^{\perp}\|^{2} + \varepsilon C_{N}  \|\tilde{h}^{\perp}\|  \|\tilde{h} \|_{L^{2}(\mathcal{M})} \\
 & - & \varepsilon^{3/2}\kappa_{2} \|m_{s}\|^{2} + \varepsilon^{3/2}C_{N} (\| e_{s}\| + \| \tilde{h}^{\perp}\|) \|\tilde{h} \|_{L^{2}(\mathcal{M})} \\
 & - &  \varepsilon^{7/4}\kappa_{3} \|w_{s}\|^{2} + \varepsilon^{7/4}C_{N} (\| e_{s} \|^{2} + \|\tilde{h}^{\perp}\|^{2} + \|m_{s}\| \|\tilde{h} \|_{L^{2}(\mathcal{M})}) \\
 & - & \varepsilon^{15/8} \| \Omega^{-1/2}\partial_{t}w_{s}  \|^{2} + \varepsilon^{15/8} C_{N} \|w_{s} \| \|\tilde{h} \|_{L^{2}(\mathcal{M})}.
\end{eqnarray*}

We now have to give an upper bound to  the unsigned terms in such a way that the negative terms absorb them, for $\varepsilon$ sufficiently small. In this purpose we use  Young inequality on every unsigned term:

$$
 \varepsilon C_{N}  \|\tilde{h}^{\perp}\|  \|\tilde{h} \|_{L^{2}(\mathcal{M})} \leq \dfrac{1}{2} \|\tilde{h}^{\perp}\|^{2} + \dfrac{1}{2} \varepsilon^{2} C_{N}^{2} \|\tilde{h} \|_{L^{2}(\mathcal{M})}^{2},
$$

$$
\varepsilon^{3/2}C_{N} \| \tilde{h}^{\perp}\| \|\tilde{h} \|_{L^{2}(\mathcal{M})} \leq  \dfrac{1}{2} \|\tilde{h}^{\perp}\|^{2} + \dfrac{1}{2} \varepsilon^{3} C_{N}^{2} \|\tilde{h} \|_{L^{2}(\mathcal{M})}^{2},
$$

$$
\varepsilon^{3/2}C_{N} \| e_{s}\|  \|\tilde{h} \|_{L^{2}(\mathcal{M})} \leq \dfrac{1}{2}\varepsilon \kappa_{1} \| e_{s}\|^{2} + \dfrac{1}{2} \varepsilon^{2} \dfrac{C_{N}^{2}}{\kappa_{1}}\|\tilde{h} \|_{L^{2}(\mathcal{M})}^{2},
$$

$$
\varepsilon^{7/4}C_{N} \| m_{s}\|  \|\tilde{h} \|_{L^{2}(\mathcal{M})} \leq \dfrac{1}{2}\varepsilon^{3/2} \kappa_{2} \| m_{s}\|^{2} + \dfrac{1}{2} \varepsilon^{2} \dfrac{C_{N}^{2}}{\kappa_{2}}\|\tilde{h} \|_{L^{2}(\mathcal{M})}^{2},
$$

$$
\varepsilon^{15/8}C_{N} \| w_{s}\|  \|\tilde{h} \|_{L^{2}(\mathcal{M})} \leq \dfrac{1}{2}\varepsilon^{7/4} \kappa_{3} \| w_{s}\|^{2} + \dfrac{1}{2} \varepsilon^{2} \dfrac{C_{N}^{2}}{\kappa_{3}}\|\tilde{h} \|_{L^{2}(\mathcal{M})}^{2}.
$$

Hence, the derivative of the entropy $\mathcal{H}_{1}$ is bounded by 

\begin{eqnarray*}
\dfrac{d}{dt}\mathcal{H}_{1}(t) & \leq & -(1-C_{N}\varepsilon^{7/4}) \|\tilde{h}^{\perp}\|^{2} \\
 & - & \varepsilon \kappa_{1} \left(\dfrac{1}{2} - \dfrac{C_{N}\varepsilon^{3/4}}{\kappa_{1}}\right) \|e_{s}\|^{2} \\
 & - & \dfrac{1}{2} \varepsilon^{3/2}\kappa_{2} \|m_{s}\|^{2} \\
  & - & \dfrac{1}{2} \varepsilon^{7/4}\kappa_{3} \|w_{s}\|^{2} \\
 & - & \varepsilon^{15/8} \| \Omega^{-1/2}\partial_{t}w_{s}  \|^{2} \\
  & + & \dfrac{C_{N}^{2}}{2} \varepsilon^{2} \left(1+\varepsilon + \dfrac{1}{\kappa_{1}} + \dfrac{1}{\kappa_{2}} + \dfrac{1}{\kappa_{3}}\right) \|\tilde{h} \|_{L^{2}(\mathcal{M})}^{2}.
\end{eqnarray*}

\noindent We now choose 

$$
0 < \varepsilon := \varepsilon_{N} < \min\left(1,\left(\dfrac{4C_{N}}{\kappa_{1}}\right)^{-4/3},  \left(\dfrac{4C_{N}}{3}\right)^{-4/7}  \right)
$$

\noindent and $\kappa_{0} = \dfrac{1}{4}$, $\kappa = \min\left(\dfrac{1}{4},\dfrac{\kappa_{1}}{4},\dfrac{\kappa_{2}}{2},\dfrac{\kappa_{3}}{2},1   \right)  $ and $C_{0,N} =   \dfrac{C_{N}^{2}}{2}  \left(2 + \dfrac{1}{\kappa_{1}} + \dfrac{1}{\kappa_{2}} + \dfrac{1}{\kappa_{3}}\right)$, and we get the result of Proposition \ref{proposition/4}.

\end{proof}

We then define a complete entropy

\begin{equation}
\mathcal{H}_{2}(t) = \mathcal{H}_{1}(t) - \varepsilon^{\frac{62}{32}} b'b - \varepsilon^{\frac{62}{32}} c'c'' 
\end{equation}

and dissispation 

\begin{equation}
\mathcal{D}_{2}(t) = \mathcal{D}_{1}(t) + b^{'2} + c^{''2}.
\end{equation}

Note that in the case of the quadratic potential, both entropy and dissipation are the same. So $\mathcal{H}_{2}$ and $\mathcal{D}_{2}$ are introduced to deal with non-quadratic cases.
We now show a last equivalence.

\begin{proposition}
There are positive constants $\Lambda_{1,N},\Lambda_{2,N},\Lambda_{3,N}$ such that for $\varepsilon_{N}>0$ small enough,

$$
\|\tilde{h} \|_{L^{2}(\mathcal{M})}^{2} \leq \Lambda_{1,N} \mathcal{H}_{2}(t) \leq \Lambda_{2,N} \mathcal{D}_{2}(t) \leq \Lambda_{3,N} \|\tilde{h} \|_{L^{2}(\mathcal{M})}^{2}.
$$

\label{proposition/6}

\end{proposition}

\begin{proof}
\noindent From the definitions, we immediately know that 

$$
\|\tilde{C}_{0}\| + \| \tilde{C}_{1} \| + \|\tilde{C}_{2} \| +\|\tilde{h}^{\perp} \|_{L^{2}(\mathcal{M})} + |b| + |c| \lesssim \|\tilde{h} \|_{L^{2}(\mathcal{M})}.
$$

\noindent By definition (\ref{equation/15}), $\|e_{s}\|  \leq \|\tilde{C}_{2} \| + |c| \lesssim \|\tilde{h} \|_{L^{2}(\mathcal{M})}$. Using  System (\ref{equation/schema_discret_2})  and doing an integration by parts, we find

$$
|c'| = \sqrt{2} |<\partial_{x}\tilde{C}_{1} >| = \sqrt{2} |<\tilde{C}_{1}, \partial_{x}\phi >| \lesssim \|\tilde{h} \|_{L^{2}(\mathcal{M})}.
$$

\noindent Using the definition of $m_{s}$ (\ref{equation/14}), 

$$
\|m_{s}\| \leq \|\tilde{C}_{1}\| + \| x\| |<\partial_{x}\tilde{C}_{1}>| + |<\tilde{C}_{1}>| \lesssim \|\tilde{h} \|_{L^{2}(\mathcal{M})}.
$$

\noindent We use the scheme (\ref{equation/schema_discret_2})  to show that $|b'| = |<\partial_{x}\tilde{C}_{0} >| \lesssim \|\tilde{h} \|_{L^{2}(\mathcal{M})}$, and, by the  definition of $w_{s}$ (\ref{equation/16}),  we get that

$$
w_{s} = \tilde{C}_{0} + b' x - \dfrac{1}{2} <\tilde{C}_{0} ( (\partial_{x}\phi)^{2} - \partial_{x}^{2}\phi )> \xi_{2} - \sqrt{2} c \phi_{s}.
$$

\noindent It is now clear that $\|w_{s}\|\lesssim \|\tilde{h} \|_{L^{2}(\mathcal{M})}$. A direct computation allows to write \\ $ c''(t) = -2<\tilde{C}_{2}\partial_{x}^{2}\phi> + \sqrt{2} <\partial_{x}^{2}\tilde{C}_{0}>$, and therefore $|c''| \lesssim \|\tilde{h} \|_{L^{2}(\mathcal{M})}$. 

\noindent Finally, we compute $\Omega^{-1/2}\partial_{t}w_{s}$:

$$
\Omega^{-1/2}\partial_{t}w_{s} = \Pi_{\mathbb{P}_N}\partial_{x}^{*}\tilde{C}_{1} - <\partial_{x}\Pi_{\mathbb{P}_N}\partial_{x}^{*}\tilde{C}_{1}> \Omega^{-1/2}x - \dfrac{1}{2} <\partial_{x}^{2}\Pi_{\mathbb{P}_N}\partial_{x}^{*}\tilde{C}_{1}>\Omega^{-1/2}\xi_{2} - \sqrt{2} c' \Omega^{-1/2}\phi_{s}.
$$

\noindent By performing integrations by parts and using (\ref{equation/H0}), we have that $\| \Omega^{-1/2}\partial_{t}w_{s} \| \leq C_{N}' \|\tilde{h} \|_{L^{2}(\mathcal{M})}$, were $C_{N}'$ is a positive constant which depends on $N$. With the definition of $\mathcal{H}_{2}(t)$ and the various estimates mentioned above, we have that 

$$
|\|h\|^{2} - \mathcal{H}_{2}(t)| \leq C_{N}' \|\tilde{h} \|_{L^{2}(\mathcal{M})}^{2}.
$$

\noindent Similarly, from the definition of $\mathcal{D}_{2}(t)$, $\mathcal{D}_{2}(t) \leq C_{N}'   \|\tilde{h} \|_{L^{2}(\mathcal{M})}^{2}$.

\noindent Using  Lemma \ref{lemme/4}, the first estimate of  Lemma \ref{lemme/5} and Lemma \ref{lemme/9}, we get that

$$
\| \tilde{C}_{0} \| \lesssim |b'| + |c''| + \|w_{s}\| + \|e_{s}\| + \|\tilde{h}^{\perp}\|.
$$

\noindent Similarly, starting from the second equality of  Lemma \ref{lemme/4} and Lemma \ref{lemme/7}

$$
\| m\| \lesssim \|\Omega^{-1/2}\partial_{t}w_{s}\| + \|m_{s}\| + |c'| + \|\tilde{h}^{\perp}\|.
$$

\noindent Finally, using the definition of $\mathcal{D}_{2}$,  that $\|\tilde{h} \|_{L^{2}(\mathcal{M})} \leq \|\tilde{C}_{0}\| + \|\tilde{C}_{1}\| + \|\tilde{C}_{2}\| + \|\tilde{h}^{\perp}\|$ and using the last three estimates together with the estimate for $c$ from  Lemma \ref{lemme/9}, we obtain that

$$
\|\tilde{h} \|_{L^{2}(\mathcal{M})} \lesssim |\mathcal{D}_{2}|.
$$

\noindent We have now shown all the desired  equivalences.

\end{proof}

Finally, we conclude this section by the proof of the main result.

\begin{proof}[\textbf{Proof of Theorem \ref{theoreme/main}}]
We use Proposition  \ref{proposition/4} and the equivalence given by Proposition \ref{proposition/6} to find that

$$
\dfrac{d}{dt}\mathcal{H}_{2}(t) \leq -\varepsilon^{62/32}\left(  \dfrac{\Lambda_{1,N}}{\Lambda_{2,N}} - C_{0,N} \Lambda_{1,N} \varepsilon^{1/16}   \right)  \mathcal{H}_{2}(t).
$$

\noindent We choose $\varepsilon < \varepsilon_N$  small enough and find a constant $\kappa_{N}>0$ such that 

 $$
\dfrac{d}{dt}\mathcal{H}_{2}(t) \leq -\kappa_{N} \mathcal{H}_{2}(t).
$$

\noindent Hence, by Gronwall lemma and Proposition \ref{proposition/6}, we   deduce that 

$$
\|\tilde{h}(t) \|_{L^{2}(\mathcal{M})}^{2} \leq  \Lambda_{3,N} e^{-\kappa_{N} t} \| \tilde{h}(0) \|_{L^{2}(\mathcal{M})} ^{2}.
$$

\noindent This concludes the proof.
\end{proof}

\section{Numerical experiments}
\label{section/5}
\subsection{Implementation}

We implement the scheme under the form (\ref{equation/schema_totally_discrete}). The iteration matrix is sparse and is block tridiagonal. We store it in CSR format. The computation of the orthonormal polynomials is done by using the procedure described in Annex \ref{section/orthonormalpolynomials}). The integrals $\alpha_{r,n} = \int_{\mathbb{R}} \tilde{P}_r' \tilde{P}_n\rho dx$ appearing in (\ref{equation/schema_totally_discrete}) are computed as following. For any integer $k$, $\tilde{P}_k$ is orthogonal to every polynomials of degree less or equal to $k-1$. Hence, if $r\leq n$, then $\alpha_{r,n}=0$. If $r>n$, we perform an integration by parts to get the equality

$$
\int_\mathbb{R} \tilde{P}_r' \tilde{P}_n \rho dx = -\int_\mathbb{R} \tilde{P}_r \tilde{P}_n' \rho dx + \int_\mathbb{R} \tilde{P}_r \tilde{P}_n \partial_x\phi\rho dx = \int_\mathbb{R} \tilde{P}_r \tilde{P}_n \partial_x\phi\rho dx.
$$

The last integrals can be computed as described in Annex \ref{section/orthonormalpolynomials} (see in particular expression \ref{equation/integral}).

\subsection{Harmonic potential}

 When $\phi(x)=\frac{1}{2}(x^2+\ln(2\pi))$, the projection is realized on Hermite polynomials both in space and in velocity. For this test, we set

$$h_0(x,v) = \tilde{H}_2(x)\tilde{H}_1(v) + \tilde{H}_1(x)\tilde{H}_2(v) + (\tilde{H}_0(x) + \tilde{H}_1(x))\tilde{H}_3(v)$$

The evolution of the norm  $\|\tilde{h} \|_{L^{2}(\mathcal{M})}$ is displayed on Figure  \ref{figure/1}, along with with the value of the exponential decay rate. The value of $N$ has no impact on the decay of the solution. We also draw the evolution of the $L^2(\rho)$ norms of $\tilde{C}_0$ and $\tilde{C}_1$ on Figure \ref{figure/2}. These two quantities evolves with the same period, but with a phase shift of half a period. This phenomenon is interpreted as information exchange between them, and result of the mixing property of the transport operator.

\subsection{Double well potential}

The potential is here $\phi(x)=(x-1)^{2}(x+1)^{2}$. We also normalize $\phi$ such that $\rho$ is probability density function. For this test, we set

$$
h_0(x,v) = \tilde{P}_2(x)\tilde{H}_0(v)+ \tilde{P}_0(x)\tilde{H}_1(v) -\sqrt{2}<\phi, \tilde{P}_2> \tilde{P}_0(x)\tilde{H}_2(v) + \tilde{P}_1(x)\tilde{H}_3(v)
$$

The evolution of the norm  $\|\tilde{h} \|_{L^{2}(\mathcal{M})}$ is displayed on Figure  \ref{figure/3}. This time, the decay of the solution is constituted of two phases. The first phase ($t\leq 20$) is characterized by a fast decay rate and oscillations, whether the second phase ($t\geq 20$) presents a much slower rate and little to no oscillations. This phenomenon has already been observed in a similar context in \cite{blaustein2024discrete}. We offer no precise explanation for this, but a finer spectral analysis of both the PDE and the iteration matrix may bring answers.  Again, the value of $N$ has no impact on the values of these rates.

We also draw the evolution of the $L^2(\rho)$ norms of $\tilde{C}_0$, $\tilde{C}_1$, $\tilde{C}_2$ on Figure \ref{figure/4}. We clearly see that the norms oscillates during the first phase and stop oscillating in the second.

\begin{figure}
\centering
\includegraphics[scale=0.55]{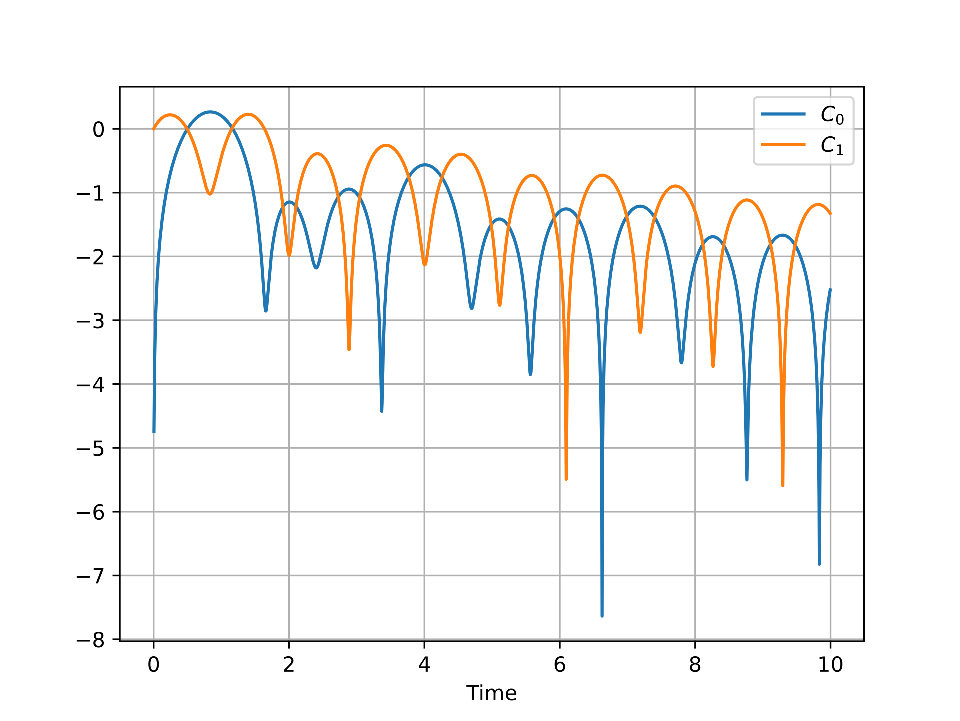}
\caption{Evolution of the norms $\|\tilde{C}_0\|$ (blue) and  $\|\tilde{C}_1\|$ (orange). The y-axis is in logscale. }
\label{figure/2}
\end{figure}

\begin{figure}
\centering
\begin{subfigure}[l]{0.49\textwidth}
\includegraphics[width=\textwidth]{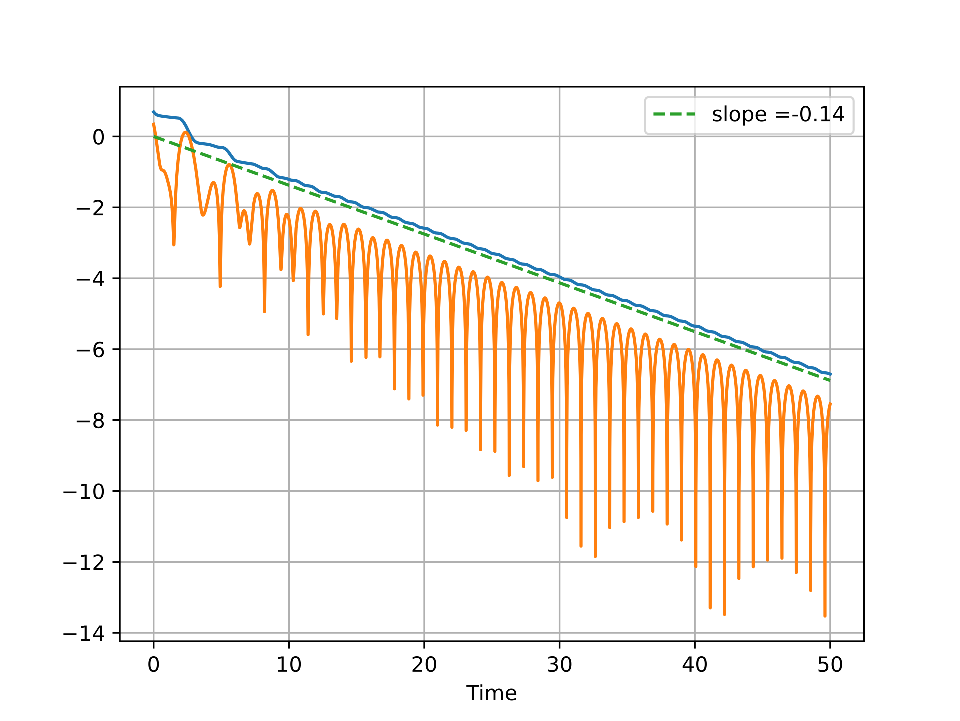}
\subcaption{ $N=15, \ \ K=20$.}
\end{subfigure}
\hfill
\begin{subfigure}[r]{0.49\textwidth}
\includegraphics[width=\textwidth]{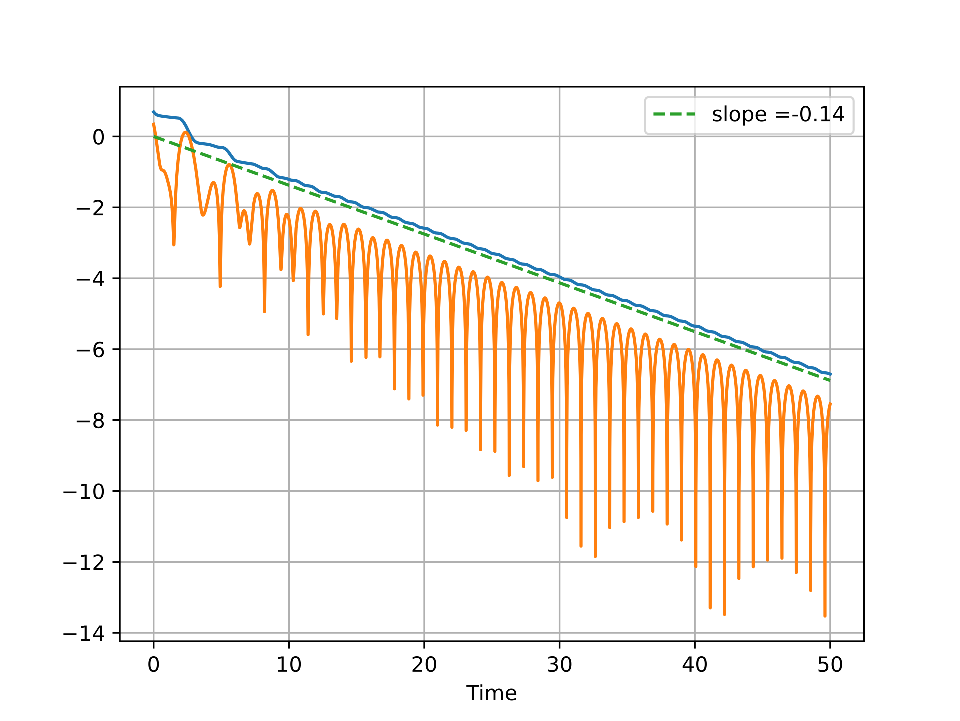}
\subcaption{$N=30, \ \ K=20$}
\end{subfigure}
\caption{Exponential decay $\|\tilde{h} \|_{L^{2}(\mathcal{M})}$ (blue) and $\|(I-\Pi)\tilde{h} \|_{L^{2}(\mathcal{M})}$ (orange, jagged) for different parameters $K$ and $N$ (y-axis  in logscale)}
\label{figure/1}
\end{figure}

\begin{figure}
\begin{subfigure}[l]{0.49\textwidth}
\includegraphics[ width=\textwidth]{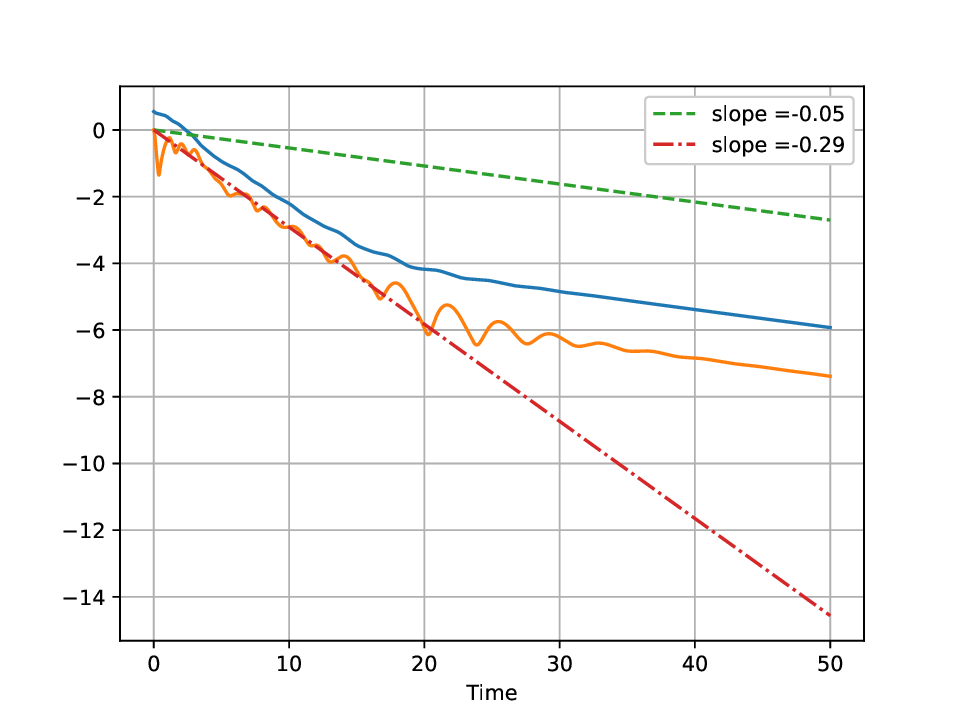}
\subcaption{ $N=15, \ \ K=20$.}
\end{subfigure}
\hfill
\begin{subfigure}[r]{0.49\textwidth}
\includegraphics[width=\textwidth]{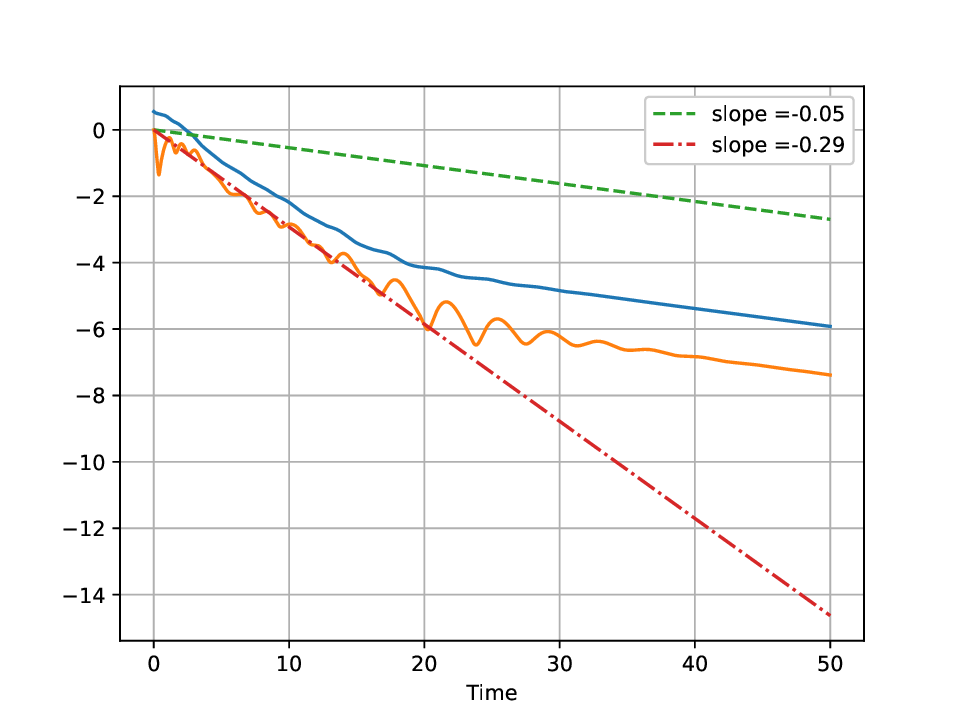}
\subcaption{$N=30, \ \ K=20$}
\end{subfigure}
\caption{Exponential decay $\|\tilde{h} \|_{L^{2}(\mathcal{M})}$ (blue) and $\|(I-\Pi)\tilde{h} \|_{L^{2}(\mathcal{M})}$ (orange) for different parameters $K$ and $N$ (y-axis  in logscale)}
\label{figure/3}
\end{figure}

\begin{figure}
\centering
\includegraphics[scale=0.55]{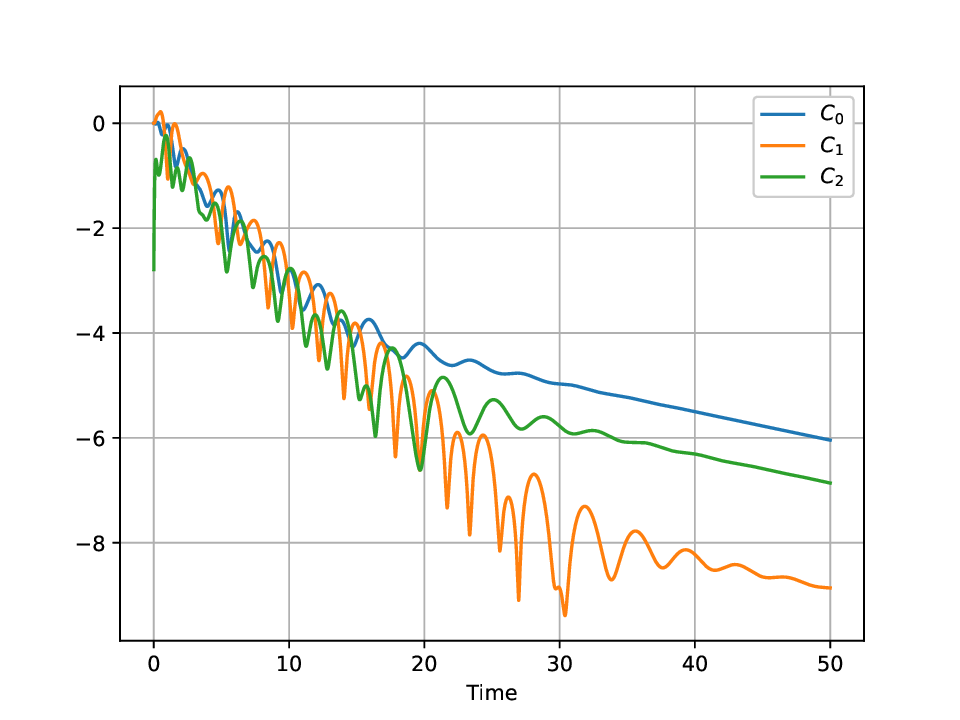}
\caption{Evolution of the norms $\|\tilde{C}_0\|$ (blue),  $\|\tilde{C}_1\|$ (orange), $\|\tilde{C}_2\|$ (green). The y-axis is in logscale. }
\label{figure/4}
\end{figure}

\ \newpage

\section{Conclusions}
\label{section/6}

In this article, we design a numerical scheme which preserves several invariants as well as the hypocoercivity of the operators at play. It is meshless, and demands no truncation of the unbounded spatial domain.

Several questions arise from this work. The first relates to the constants appearing in Theorem \ref{theoreme/main}. They may be unbounded with respect to the truncation parameter $N$. In fact, they are bounded with respect to $N$ if the constant $K_N$, defined by inequalities (\ref{equation/H0})-(\ref{equation/H3}) is itself bounded with respect to $N$.  When $\phi$ is harmonic, $K_N$ is indeed bounded, and the proof is elementary.  We conjecture that $K_N$ is always bounded, regardless of the potential.

\begin{conjecture}
 There exists a constant $K>0$ independant of $N$ such that:

\begin{equation}
\sup_{f\in \mathbb{P}_N,f\neq 0} \dfrac{\| \Omega^{-1/2} \Pi_{\mathbb{P}_N} \partial_{x}^{*}f  \|}{\|f\|} \leq K,
\label{equation/C0}
\end{equation}

\begin{equation}
\sup_{f\in \mathbb{P}_N,f\neq 0} \dfrac{\| \Omega^{-1}\partial_{x} \Pi_{\mathbb{P}_N} \partial_{x}^{*}f  \|}{\|f\|} \leq K,
\label{equation/C1}
\end{equation}

\begin{equation}
\sup_{f\in \mathbb{P}_N,f\neq 0} \dfrac{\| \Omega^{-1} \Pi_{\mathbb{P}_N} \partial_{x}^{*} \Pi_{\mathbb{P}_N} \partial_{x}^{*}f  \|}{\|f\|} \leq K,
\label{equation/C2}
\end{equation}

\begin{equation}
\sup_{f\in \mathbb{P}_N,f\neq 0} \dfrac{\| \Omega^{-1} \Pi_{\mathbb{P}_N} \partial_{x}^{*}\partial_{x}f  \|}{\|f\|} \leq K.
\label{equation/C3}
\end{equation}
\label{conjecture}
\end{conjecture}

 A proof may appear in later work. A second question concerns the closure chosen for System (\ref{equation/schema_discret_1}).  It is possible that an other closure would be more appropriate in term of accuracy. Finally, the transition between two regimes during the relaxation toward equilibrium, which have been observed in our simulations, is not predicted by theoretical result.  It might be of great interest to investigate this phenomenon in the continuous framework.

\appendix

\section{Orthonormal polynomials }

\label{section/orthonormalpolynomials}

In this appendix, we gather definitions, notations and important results on the orthonormal polynomials sequences associated with a weight of the form $\rho(x)=e^{-\phi(x)}$, normalized such that $\int_{\mathbb{R}}\rho dx = 1$. The potential $\phi$ is an even polynomial of degree $2m$ of positive leading coefficient $\gamma_{m}$. It can be expanded in the canonical basis:

$$
\phi(x) = \sum_{i=0}^{m} \gamma_{i} x^{2i}
$$

The orthonormal polynomials  $(\tilde{P}_{n})_{n\in\mathbb{N}}$ constitute a Hilbert basis  of $L^2(\rho)$.  They satisfy a three terms recurrence relation:

\begin{equation}
\left\lbrace
\begin{array}{lll}
x\tilde{P}_{n}(x) & = & a_{n} \tilde{P}_{n+1}(x) +  a_{n-1} \tilde{P}_{n-1}(x) \ \forall n\geq 0, \\
\tilde{P}_{0} &=& 1  \\ 
\tilde{P}_{-1} &=& 0 
\end{array}
\right.
\label{equation/recurrence_polynome_2}
\end{equation}

The coefficients $a_n$ are positive, and $a_{-1}$ can be set arbitrarily. The coefficients are functions only of the weight $\rho$.

\begin{remark} In the special case where $\phi(x) = \frac{1}{2}(x^{2}+\ln(2\pi))$, the orthonormal polynomials are called the Hermite orthonormal polynomials. They are defined by the following recurrence relation:

$$\left\lbrace
\begin{array}{lll}
 v\tilde{H}_{k}(v)& = &\sqrt{k+1} \tilde{H}_{k+1}(v) + \sqrt{k} \tilde{H}_{k-1}(v) \ \ \forall k \geq 1, \\
 \tilde{H}_{0} &=& 1 \\ 
 \tilde{H}_{-1} &=& 0 .
\end{array}\right.
$$

They enjoy a richer structure, and satisfy an important differential property:

$$
\tilde{H}^{'}_{k}(v)= \sqrt{k}\tilde{H}_{k-1}(v) \ \forall k \in \mathbb{N}.
$$
\end{remark}

The computation of the orthonormal polynomials thus demands to compute the coefficients $a_k$ for all $k\in\mathbb{N}$. A first approach is to use Freud's equation, which gives an induction relation on the coefficients $a_k$. First, one can use the three-terms recurrence relation to compute the integral

\begin{equation}
\int_{\mathbb{R}} \tilde{P}_{k+1}\tilde{P}_k \partial_x\phi \rho dx.
\label{equation/integral}
\end{equation}

Indeed, one can use the recurrence relation to expand $\partial_x\phi\tilde{P}_{k+1}$ on the basis $(\tilde{P}_{n})_{n\in\mathbb{N}}$. Then, it remains to identify the component along $\tilde{P}_k$. For an even quartic potential $\phi$, this gives

$$
\int_{\mathbb{R}} \tilde{P}_{k+1}\tilde{P}_k \partial_x\phi \rho dx = 2\gamma_2 a_k + 4\gamma_4 a_k(a_{k-1}^2+a_k^2 + a_{k+1}^2)
$$

An other way to express this integral is by performing an integration by parts. First, $\tilde{P}_{k+1}$is orthogonal to every polynomial of degree less than $k$, thus

\begin{equation*}
 \int_{\mathbb{R}} \tilde{P}_k \tilde{P}_{k+1} \partial_x\phi \rho dx  =  \int_\mathbb{R}( \tilde{P}_k' \tilde{P}_{k+1} + \tilde{P}_k \tilde{P}_{k+1}') \rho dx =\int_\mathbb{R} \tilde{P}_k \tilde{P}_{k+1}' \rho dx
\end{equation*}

Secondly,  it is easy to see that the leading coefficient of $\tilde{P}_k$ is $\prod_{i=0}^{k-1} \frac{1}{a_i}$ from the three-terms recurrence relation. Hence, 

\begin{equation*}
 \int_\mathbb{R} \tilde{P}_k \tilde{P}_{k+1}' \rho dx =  (k+1)\prod_{i=0}^{k} \frac{1}{a_i}\int_\mathbb{R} \tilde{P}_k x^k  \rho dx = \frac{k+1}{a_k} \int_{\mathbb{R}} \tilde{P}_k^2 \rho dx
\end{equation*}

Finally, the last integral equals $1$ since $\tilde{P}_k$ is normalized. We have two expressions for (\ref{equation/integral}) and they lead to the so-called Freud's equation \cite{MAGNUS198665}. For an even quartic potential $\phi$, Freud's equation reads 

\begin{equation}
k+1 = 2\gamma_2 a_k^2 + 4\gamma_4 a_k^2(a_{k-1}^2+a_k^2 + a_{k+1}^2) \ \forall k\geq 1
\label{equation/Freudequation}
\end{equation}

It allows one to compute the coefficients given $a_0$ and $a_1$, which can be expressed using the second and forth order moments of $\rho$. This equation generalizes for potentials of higher degree. This method has the advantage to require the computation of a few, fixed number of moments for the computation of an arbitrary number of coefficients $a_k$, but is ill-conditionned as it may propagate and amplify small errors on the firsts values $a_0$ and $a_1$. Thus, it can't be used to compute $a_k$ for large $k$. For large $k$, we can use the asymptotic expansion of $a_k$ in powers of $k^{-\frac{1}{2m}}$ (see \cite{deift}). We know that there are explicit, real numbers $c^{(0)},^{(1)},...,c^{(4m)}$ such that 

$$
a_{k-1} = k^{\frac{1}{2m}} \sum_{l=0}^{4m} c^{(l)} k^{-\frac{l}{2m}} + O(k^{-2})
$$

We propose the following algorithm for the computation of the coefficients $a_k$. Fix a integer $K_0$. 

\begin{itemize}
\item If $k\leq K_0$,  iterate Freud's equation to compute $a_k$;
\item If $k >K_0$, use the asymptotic expansion to compute $a_k$.
\end{itemize}

We made experimentation on three differents even quartic potentials (Figure \ref{figure/freudvsexpansion} and \ref{figure/diff}). We see that the sequence given by Freud's equation diverges. Its accuracy depends strongly on the potential. When the potential in nonconvex, the asymptotic expansion is not accurate for small indices. Between these two regimes, the two sequences agree (see Figure \ref{figure/diff}). Overall, the expansion is accurate but $K_0$ must not be choosen too small, particularly when the potential is nonconvex. In our experimentations,we will set $K_0=10$.

\begin{figure}
\begin{subfigure}{0.45\textwidth}
\includegraphics[width=\textwidth]{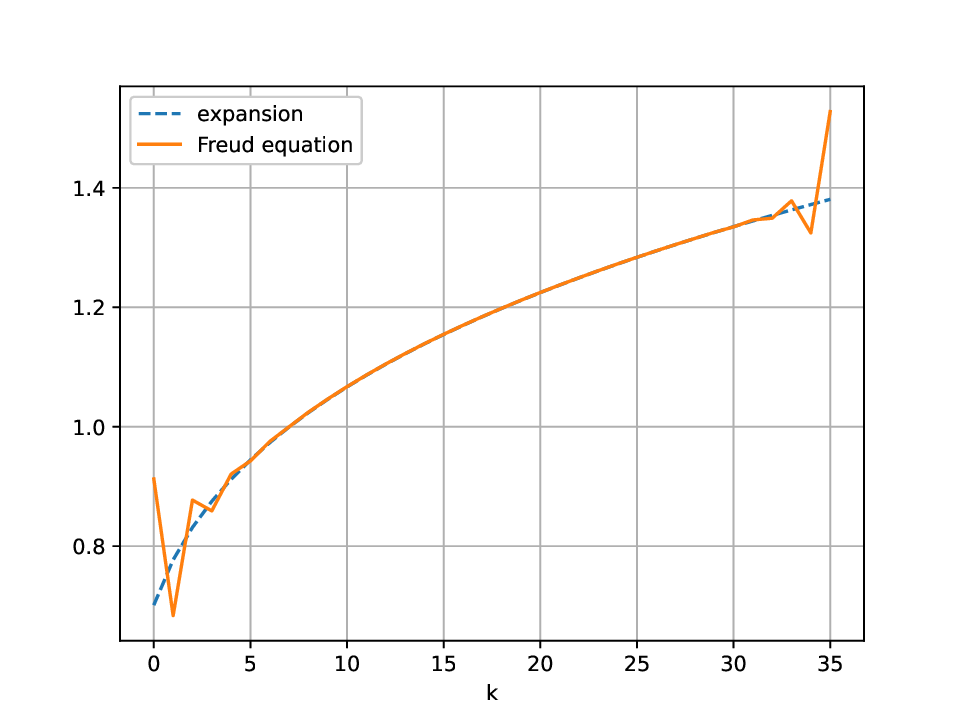}
\caption{$x^4-2x^2$}
\end{subfigure}
\hfill
\begin{subfigure}{0.45\textwidth}
\includegraphics[width=\textwidth]{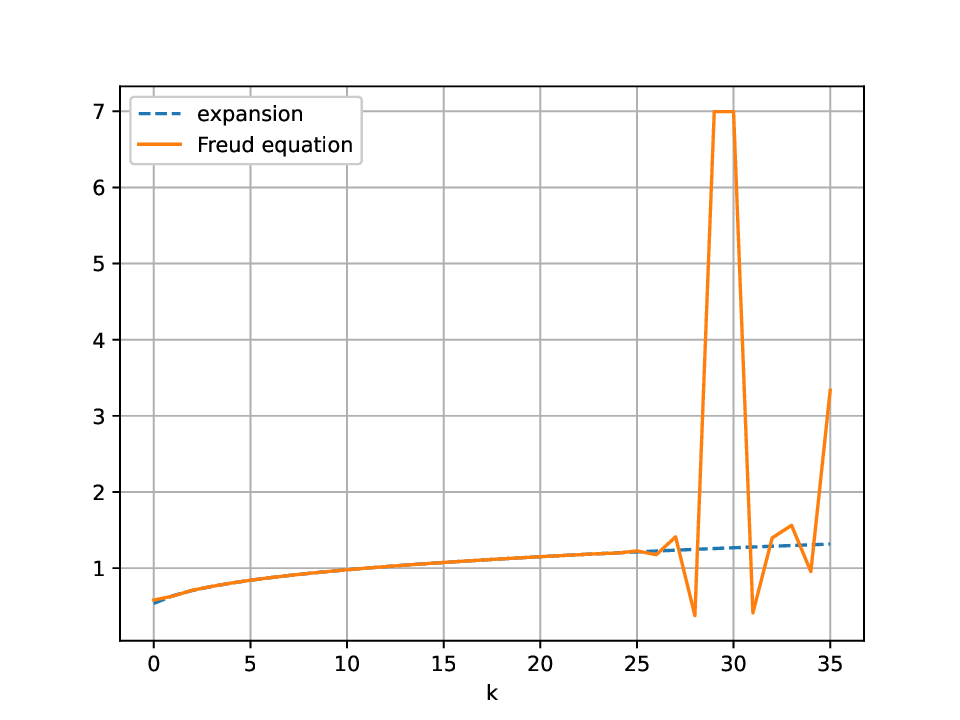}
\caption{$x^4$}
\end{subfigure}
\vfill
\begin{subfigure}{0.45\textwidth}
\includegraphics[width=\textwidth]{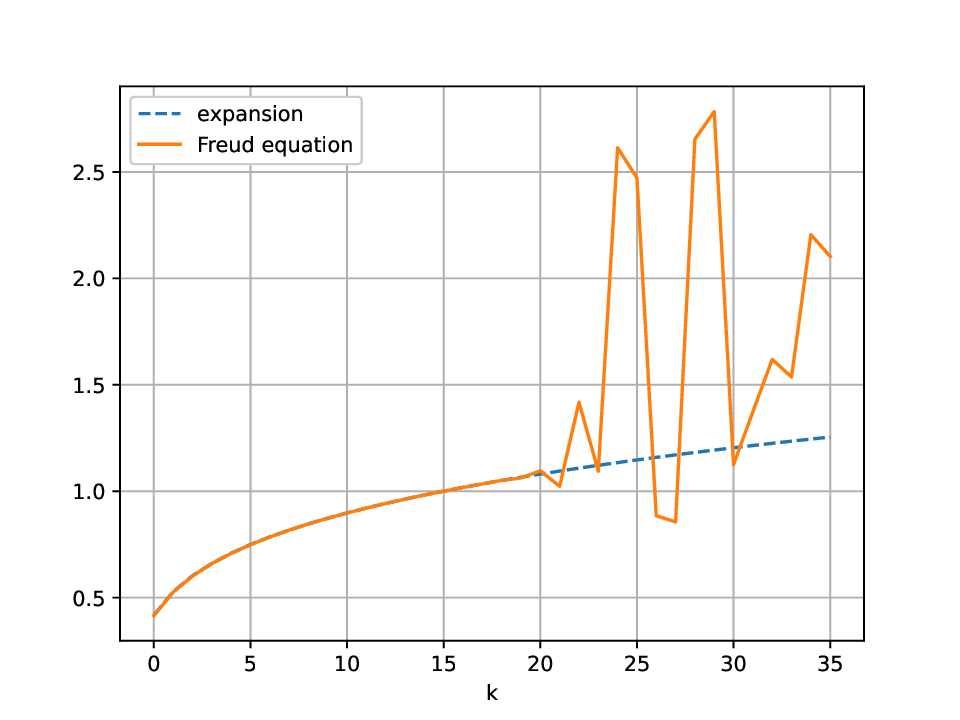}
\caption{$x^4+2x^2$}
\end{subfigure}
\caption{Comparison of the sequences given by the asymptotic expansion (dashed) and Freud's equation (solid) for different potentials $\phi$.}
\label{figure/freudvsexpansion}
\end{figure}

\begin{figure}
\begin{subfigure}{0.45\textwidth}
\includegraphics[width=\textwidth]{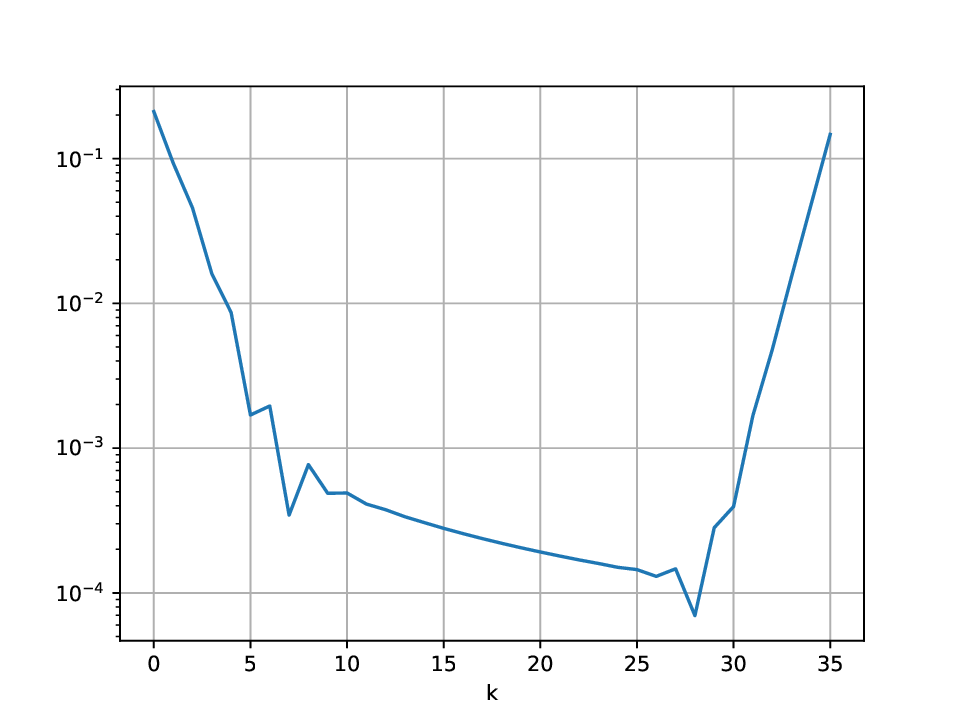}
\caption{$x^4-2x^2$}
\end{subfigure}
\hfill
\begin{subfigure}{0.45\textwidth}
\includegraphics[width=\textwidth]{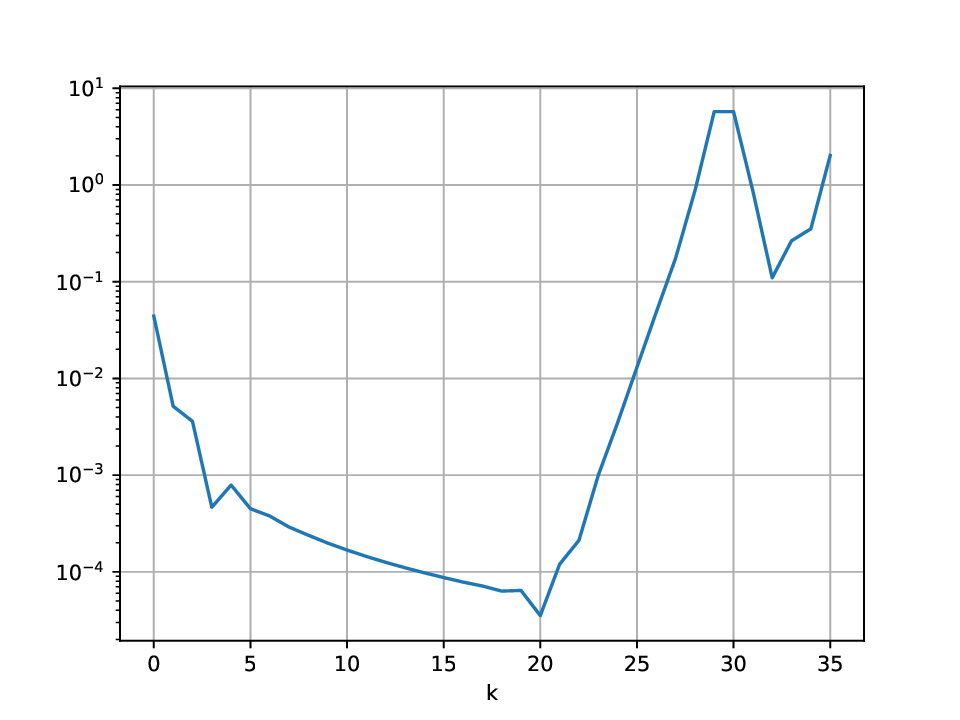}
\caption{$x^4$}
\end{subfigure}
\vfill
\begin{subfigure}{0.45\textwidth}
\includegraphics[width=\textwidth]{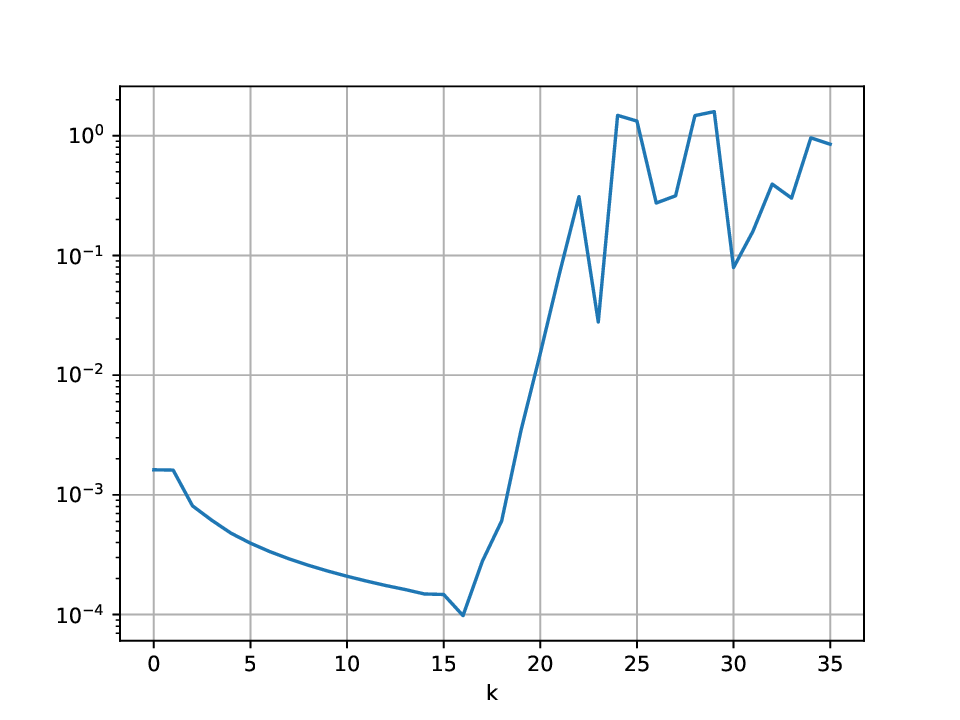}
\caption{$x^4+2x^2$}
\end{subfigure}
\caption{Difference between the coefficients generated by Freud's equation and the expansion for different potentials $\phi$.}
\label{figure/diff}
\end{figure}

\ \\\newpage
\ \\ \newpage

\bibliographystyle{plain}
\bibliography{biblio}

\end{document}